\documentclass[a4paper,11pt]{amsart}

\usepackage[utf8]{inputenc}
\usepackage{lmodern}
\usepackage{url}
\usepackage{amsmath, amsthm, amssymb, amscd, accents, bm}
\usepackage{mathtools}
\usepackage{mdwlist}
\usepackage{paralist}
\usepackage{graphicx}
\usepackage{epstopdf}
\usepackage{datetime}
\usepackage{afterpage}
\usepackage[dvipsnames]{xcolor}
\usepackage{hyperref}
\usepackage[ruled,vlined]{algorithm2e}

\usepackage{caption}
\usepackage{subcaption}

\usepackage[foot]{amsaddr}

\usepackage{mathtools}
\DeclarePairedDelimiter{\ceil}{\lceil}{\rceil}
\DeclarePairedDelimiter\floor{\lfloor}{\rfloor}

\newcommand\rurl[1]{%
  \href{https://#1}{\nolinkurl{#1}}%
}

\newtheorem{definition}{Definition}[section]
\newtheorem{theorem}[definition]{Theorem}
\newtheorem{lemma}[definition]{Lemma}
\newtheorem{corollary}[definition]{Corollary}
\newtheorem{proposition}[definition]{Proposition}

\newtheorem{remark}[definition]{Remark}

\newtheorem*{theorem*}{Theorem}

\def\N{{\mathbb N}}
\def\Z{{\mathbb Z}}
\def\R{{\mathbb R}}
\def\T{{\mathbb T}}
\def\C{{\mathbb C}}

\newcommand{\supp}{{\mathrm{supp}}}

\newcommand{\lt}{{L^2(\R)}}

\newcommand{\bone}{{\boldsymbol 1}}

\newcommand{\ift}{{\mathcal{F}^{-1}}}
\newcommand{\ft}{{\mathcal{F}}}

\makeatletter
\newcommand\thankssymb[1]{\textsuperscript{\@fnsymbol{#1}}}

\makeatletter
\def\@makefnmark{%
  \leavevmode
  \raise.9ex\hbox{\fontsize\sf@size\z@\normalfont\tiny\@thefnmark}}

\begin{document}

\title[Phase retrieval in Gaussian shift-invariant spaces]{Stable Gabor phase retrieval in Gaussian shift-invariant spaces via biorthogonality}
\author[Philipp Grohs]{Philipp Grohs\thankssymb{1}\textsuperscript{,}\thankssymb{2}\textsuperscript{,}\thankssymb{3}}
\address{\thankssymb{1}Faculty of Mathematics, University of Vienna, Oskar-Morgenstern-Platz 1, 1090 Vienna, Austria}
\address{\thankssymb{2}Research Network DataScience@UniVie, University of Vienna, Kolingasse 14-16, 1090 Vienna, Austria}
\address{\thankssymb{3}Johann Radon Institute of Applied and Computational Mathematics, Austrian Academy of Sciences, Altenbergstrasse 69, 4040 Linz, Austria}
\email{philipp.grohs@univie.ac.at, philipp.grohs@oeaw.ac.at}
\author[Lukas Liehr]{Lukas Liehr\thankssymb{1}}
\email{lukas.liehr@univie.ac.at}

\maketitle

\begin{abstract}
We study the phase reconstruction of signals $f$ belonging to complex Gaussian shift-invariant spaces $V^\infty(\varphi)$ from spectrogram measurements $|\mathcal G f(X)|$ where $\mathcal G$ is the Gabor transform and $X \subseteq \R^2$.
An explicit reconstruction formula will demonstrate that such signals can be recovered from measurements located on parallel lines in the time-frequency plane by means of a Riesz basis expansion. Moreover, connectedness assumptions on $|f|$ result in stability estimates in the situation where one aims to reconstruct $f$ on compact intervals. 
Driven by a recent observation that signals in Gaussian shift-invariant spaces are determined by lattice measurements [Grohs, P., Liehr, L., \emph{Injectivity of Gabor phase retrieval from lattice measurements}, Appl. Comput. Harmon. Anal. \textbf{62} (2023), pp. 173--193] we prove a sampling result on the stable approximation from finitely many spectrogram samples. The resulting algorithm provides a provably stable and convergent approximation technique.
In addition, it constitutes a method of approximating signals in function spaces beyond $V^\infty(\varphi)$, such as Paley-Wiener spaces.

\vspace{0.4cm}

\end{abstract}

\section{Introduction}

Phaseless signal reconstruction deals with the problem of recovering a function $f$ from intensity measurements of the form $\{ |Af(x)| : x \in X \}$ where $A$ is a linear transformation and $X$ is a subset of the domain of $Af$. This is known as the phase retrieval problem, a non-linear inverse problem arising in numerous applications in science and engineering, such as coherent diffraction imaging \cite{Silva,Fienup,Shechtman}, speech recognition \cite{Couvreur} and quantum mechanics \cite{CORBETT200653}. In the situation where $f$ belongs to an infinite-dimensional Banach or Hilbert space, phase retrieval is known to be never uniformly stable \cite{alaifariGrohs, daubcahill}, hence constituting a challenging problem for numerical reconstruction approaches. See the surveys \cite{Strohmer,GrohsKoppensteinerRathmair} for a current overview on uniqueness, stability and algorithms for phase retrieval. Recent attention was drawn to signals having a shift-invariant structure, that is,
\begin{equation}\label{def:si_space}
f \in V_\beta^p(\phi) = \left \{ \sum_{n \in \Z} c_n \phi(\cdot - \beta n) : (c_n)_n \in \ell^p(\Z) \right \},
\end{equation}
where $\phi \in L^p(\R)$ is the so-called generator of $V_\beta^p(\phi)$ and $p \in [1,\infty]$. The constant $\beta > 0$ is a regularity parameter, commonly referred to as the step-size. In a series of papers, Chen et. al. studied the phase retrievability of a real-valued map $f$ from its modulus $|f|, A = \mathrm{Id}$, assuming that the generator has compact support \cite{chen1,chen2,chen3}. Shenoy et. el. investigated the situation where $A =  \ft$ is the Fourier transform and the defining sequence $(c_n)_n$ of $f$ obeys special properties \cite{shenoy}.
In the situation where $\phi = \mathrm{sinc}$ is the cardinal sine, i.e. the considered signals lie in the classical Paley-Wiener space, Thakur demonstrated how to recover a real-valued band-limited function from unsigned samples lying on a suitable dense sampling set \cite{Thakur2011}.
Similar results for real-valued maps were derived for Gaussian generators as well as totally positive generators of Gaussian type and are due to Gröchenig \cite{groechenigPhase} and Romero \cite{romeroPhase}. Besides, it was shown that if $\phi = \varphi^\sigma$ is the Gaussian
\begin{equation}\label{def:gaussian}
\varphi^\sigma(t) = e^{-\frac{t^2}{2\sigma^2}}
\end{equation}
with variance $\sigma^2$ then, in general, complex-valued maps are not determined from samples of their absolute value. Hence, additional information is needed to achieve injectivity for complex signals from measurements $|Af|$ with a suitably chosen signal-transformation $A$. The authors of the present article proved in \cite{grohsliehr2020} that every complex-valued map in $V_\beta^1(\varphi)$ is uniquely determined up to a global phase factor from measurements of the form $|\mathcal{V}_gf(X)|$ where $\mathcal{V}_g$ is the short-time Fourier transform with Gaussian window $g=\varphi^\sigma$,
\begin{equation}\label{def:gabor}
    \mathcal{V}_gf(x,\omega) = \int_\R f(t)\overline{g(t-x)}e^{-2\pi i \omega t} \, dt,
\end{equation}
and $X \subset \R^2$ is a separated set of sampling points in the time-frequency plane. The resulting phase retrieval problem with transformation $A=\mathcal{V}_g$ is known as the Gabor phase retrieval problem which plays an important role in applications such as ptychography, a popular and highly successful approach in coherent diffraction imaging.
The previous observations serve as the starting point for a more penetrating investigation of the Gabor phase retrieval problem in a Gaussian shift-invariant setting and motivate the work in hand. In particular, we derive an explicit inversion formula and stability estimates as well as propose a new reconstruction algorithm which approximates complex-valued signals in $V_\beta^\infty(\varphi)$ from finitely many (deterministic) samples of $|\mathcal{V}_gf|$ in a provable and stable manner.

Classical algorithmic approaches on solving the phase retrieval problem are based on fixed point iterations, iterative projection methods and gradient descent methods.
Examples include the alternating projection method, the Gerchberg–Saxton algorithm \cite{Gerchberg1972APA}, error reduction \cite{Fienup1982PhaseRA}, the Douglas-Rachford algorithm \cite{Fannjiang2020FixedPA}, Hybrid Input-Output Algorithm \cite{Fienup1982PhaseRA} and the Averaged Alternating Reflection Algorithm \cite{Luke_2004,Li_2017}.
The difficulty of analyzing convergence and robustness of these methods lies in the non-convexity of the phase retrieval problem. Heuristically, these methods perform well if one chooses a right initialization which is close to a solution of the problem. However, it is in general unclear how to perform the initialization step and provide convergent results in a noise regime. Moreover, the discretization of these methods in an infinite-dimensional setting constitutes a difficult task.
In contrast, the derived algorithm of the present paper overcomes these difficulties by providing a provably convergent approximation routine from noisy samples. Moreover, we quantify exactly how many samples of $|\mathcal{V}_gf|$ need to be taken in order to guarantee a given reconstruction error.

\subsection{Notation}

Throughout this article, the shift-invariant space $V_\beta^p(\phi)$, the Gaussian $\varphi^\sigma$ and the short-time Fourier transform are defined as in \eqref{def:si_space}, \eqref{def:gaussian} and \eqref{def:gabor}, respectively. The cross-ambiguity function of $f$ and $g$ is the map $A(f,g)(x,\omega) = e^{\pi i x \omega}\mathcal{V}_gf(x,\omega)$. The spectrogram of $f$ w.r.t. to a window $g$ is given by the map $(x,\omega) \mapsto |\mathcal{V}_gf(x,\omega)|^2$. If $g = \varphi^\sigma$ is a Gaussian then we set
$$
\mathcal{G} \coloneqq \mathcal{V}_{\varphi} \coloneqq \mathcal{V}_{\varphi^\sigma}
$$
and call $\mathcal{G}$ the Gabor transform of $f$.
For clarity of the exposition and simplicity of notation, we will drop the dependence on $\sigma$.
For a function $f : \R \to \C$ and an $\omega \in \R$ we define its tensor product $f_\omega$ via
\begin{equation}\label{def:tensor}
   f_\omega \coloneqq (T_\omega f)\overline{f}
\end{equation}
where $T_\omega f = f(\cdot - \omega)$ is the shift-operator. Moreover, the modulation operator $M_\omega$ is defined by $M_\omega f(t) = e^{2\pi i \omega t}f(t)$. In Section \ref{section2} we will exhibit that under suitable assumptions on $\phi$ the system of translates $(T_{\beta n} \phi)_{n \in \Z}$ forms a Riesz basis for the space $V_\beta^2(\phi)$. Denoting by $S$ the frame operator associated with $(T_{\beta n} \phi)_{n \in \Z}$ and making use of the property that the operators $S$ and $T_{\beta n}$ commute for every $n \in \Z$ implies that there exists a map $\tilde \phi$ such that $(T_{\beta n} \tilde \phi)_{n \in \Z}$ is the dual frame of $(T_{\beta n} \phi)_{n \in \Z}$. The function $\tilde \phi$ given by
\begin{equation}
    \tilde \phi \coloneqq S^{-1}\phi
\end{equation}
is called the dual generator of $\phi$.
If $(T_{\beta n} \phi)_{n \in \Z}$ forms a Riesz basis for $V_\beta^2(\phi)$ and $f =  \sum_n c_n T_{\beta n } \phi$ then we call $c = (c_n)_n \subset \C$ the defining sequence of $f$.
In order to study stability properties of the Gabor phase retrieval problem in a Gaussian shift-invariant regime we define a mixed norm as follow.
Let $\alpha>0, p \in [1,\infty]$ and suppose that $F: \alpha \Z \times \R \to \C$ is measurable. The mixed norm $\| \cdot \|_{\alpha,p}$ is defined by
\begin{equation}\label{def:mixed_norm}
\| F \|_{\alpha,p} \coloneqq \begin{cases} 
      \left ( \sum_{n \in \Z} \| F(\alpha n, \cdot) \|_{L^1(\R)}^p \right )^{\tfrac{1}{p}}, & p \in [1,\infty) \\
      \sup_{n \in \Z} \| F(\alpha n, \cdot) \|_{L^1(\R)},  & p = \infty
   \end{cases}.
\end{equation}
This is the $\ell^p$-norm of the sequence $(\| F(\alpha n, \cdot) \|_{L^1(\R)})_{n \in \Z}$.
If $F$ is defined on $\R^2$ we set $\| F \|_{\alpha,p} \coloneqq \| F|_{\alpha\Z \times \R} \|_{\alpha,p}$ where $F|_{\alpha\Z \times \R}$ is the restriction of $F$ to $\alpha\Z \times \R$. If not specified otherwise, all indices will run over the integers $\Z$. Finally, we say that two functions $f,h : \R \to \C$ equal up to a global phase (or: up to a unimodular constant) if there exists a $\tau \in \T \coloneqq \{ z \in \C : |z|=1 \}$ such that $f=\tau h$.

\subsection{Main results}

The first main result, Theorem \ref{thmA:explicit_reconstruction}, provides an explicit inversion formula of a function $ f \in V_\beta^\infty(\varphi)$ from its spectrogram $|\mathcal{G}f(X)|$ where $X = \frac{\beta}{2}\Z \times \R$ is a set of vertical lines in the time-frequency plane. 

\begin{theorem}\label{thmA:explicit_reconstruction}
Let $f \in V_\beta^\infty(\varphi)$ and let $p \in \R$ such that $f(p) \neq 0$. Then there exists a unimodular constant $\tau \in \T$ such that
$$
f(p+\omega) = \tau C^{-\frac{1}{2}} \sum_{n \in \Z} \left ( \int_\R |\mathcal{G}f(\tfrac{\beta}{2}n,t)|^2e^{2\pi i \omega t} \, dt \right ) T_{\frac{\beta}{2}n} \widetilde{\varphi_\omega}(p+\omega)
$$
for every $\omega \in \R$ where
$
C = \sum_{n \in \Z} \left ( \int_\R |\mathcal{G}f(\tfrac{\beta}{2}n,t)|^2 \, dt \right ) T_{\frac{\beta}{2}n} \widetilde{\varphi_0}(p)
$
and $\widetilde{\varphi_\omega}$ is the dual generator of the Gaussian tensor product $\varphi_\omega$.
\end{theorem}

In the previous theorem the maps $\{ \widetilde{\varphi_\omega} : \omega \in \R \}$ constitute a parametrized system of dual generators. In particular, the statement shows that two functions $f,h \in V_\beta^\infty(\varphi)$ agree up a global phase provided that their spectrograms agree on $X = \frac{\beta}{2}\Z \times \R$. Note, that this is a uniqueness statement within the signal class $V_\beta^\infty(\varphi)$. A similar statement does not hold if one replaces $V_\beta^\infty(\varphi)$ by $\lt$, as we have shown in \cite{grohsLiehrJFAA}.

Next, we study stability properties of the above inversion. Heuristically, Theorem \ref{thmB} assumes that the map $f$ has the property that $|f|$ is not too small on large intervals. More precisely, it will be assumed that there exists $J \in \N, p_1, \dots, p_J \in \R$ and $\gamma>0$ such that
\[
J \geq 2 \ \ \land \ \ p_1 < p_2 < \cdots < p_J \ \ \land \ \ |f(p_j)| \geq \gamma \ \forall j \in \{ 1, \dots, J \} \tag{\textbf{P}}.
\]
Condition \textbf{(P)} excludes examples of the form $f^+_s = \varphi(\cdot - a) + \varphi(\cdot -b), f^-_s = \varphi(\cdot - a) - \varphi(\cdot -b), s=|a-b|$, where it is known that exponential instabilities arise for growing $s$ \cite{ALAIFARI2021401}. These examples further show that an assumption on $|f|$ of the above form cannot be dropped.

\begin{theorem}\label{thmB}
Let $\gamma > 0$, $f \in V_\beta^\infty(\varphi)$ and $p_1 < \cdots < p_J \in \R$ such that $|f(p_j)| \geq \gamma$ for all $j$. If $I = [p_1-r,p_J+r]$ and $r = \max_{1 \leq j \leq J-1} (p_{j+1}-p_j)$ then for every $g \in V_\beta^\infty(\varphi)$ we have
\begin{equation*}
\begin{split}
& \min_{\tau \in \T} \| f-\tau g \|_{L^\infty(I)} \\ & \lesssim_{\beta, \sigma} (J-1)e^{\frac{r^2}{4\sigma^2}} \frac{\max \{ \| f \|_{L^\infty(I)}^2, \| f \|_{L^\infty(I)}+\| g \|_{L^\infty(I)} \} }{\min \{ \gamma,\gamma^3 \}} \| |\mathcal{G}f|^2-|\mathcal{G}g|^2 \|_{\frac{\beta}{2},\infty}
\end{split}
\end{equation*}
and the implicit constant depends only on the step-size $\beta$ and the standard-deviation $\sigma$ of the Gaussian $\varphi=\varphi^\sigma$.
\end{theorem}

\begin{figure}
\centering
   \includegraphics[width=12.5cm]{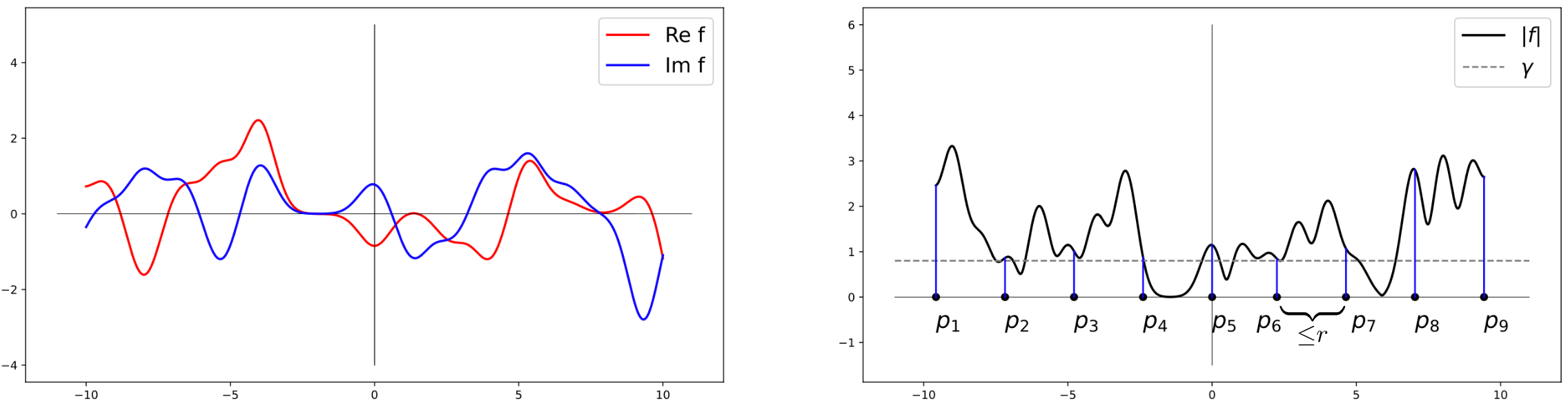}
\caption{Visualization of condition \textbf{(P)}: the points $p_1, \dots, p_9$ satisfy the properties $|f(p_j)| \geq \gamma$ for all $j=1, \dots,9$ and $p_{j+1}-p_j \leq r$ for all $j=1, \dots,8$.}
\end{figure}

We seek to discretize the previous regime in the sense that it yields an algorithmic approximation of functions in $V_\beta^\infty(\varphi)$ from finitely many measurements located on a grid of the form
\begin{equation}\label{def:X}
    X = \tfrac{\beta}{2} \{ -N, \dots, N \} \times h \{ -H, \dots, H \}
\end{equation}
where $N,H \in \N, h,\beta >0$ rather than $X=\frac{\beta}{2}\Z \times \R$. If $\eta \in \R^{(2N+1) \times (2H+1)}$ denotes a noise matrix then the given samples $\mathfrak{S}$ take the form
\begin{equation*}
    \begin{split}
        & \mathfrak{S}=(|\mathcal{G}f(X)|^2+\eta_{n,k})_{n,k} \in \R^{(2N+1) \times (2H+1)}, \\
        & \mathfrak{S}(n,k) = |\mathcal{G}f(\tfrac{\beta}{2}n,hk)|^2 + \eta_{n,k}.
    \end{split}
\end{equation*}
The contribution of the noise $\eta$ will be quantified in terms of the $\ell^\infty$-operator norm of $\eta$, i.e. the maximum absolute row sum of $\eta$. We denote this norm by $\| \eta \|_\infty$. The so-called numerical approximation routine $\mathcal{R}$ depends only on $\mathfrak{S}$ and is defined as follows: for
$
p_1 < p_2 < \cdots < p_J
$
define constants $c_j$ by
$
c_j = h\sum_n \sum_k \mathfrak{S}(n,k) T_{\frac{\beta}{2}n} \widetilde{\varphi_0}(p_j),
$
functions $L_j$ by
$$
L_j(\omega) = \frac{h}{\sqrt{c_j}} \sum_n \sum_k \mathfrak{S}(n,k)e^{2\pi i \omega h k} T_{\frac{\beta}{2}n} \widetilde{\varphi_\omega}(p_j+\omega)
$$
and phases $\nu_0, \dots, \nu_{J-1} \in \T$ by
$$
\nu_0 = 1, \ \ \nu_j = \frac{L_j(p_{j+1}-p_j)}{|L_j(p_{j+1}-p_j)|} \ (j=1, \dots, J-1).
$$
Assuming that $L_j$ and $\nu_j$ are well-defined, i.e. $c_j >0$ and $L_j(p_{j+1}-p_j) \neq 0$ we define $\mathcal{R}:[p_1,p_J] \to \C$ by
\begin{equation}
    \mathcal{R}(t) = \nu_1 \cdots \nu_{j-1} L_j(\omega),
\end{equation}
if $t \in [p_1,p_J]$ such that $t \in (p_j,p_{j+1}]$ with $t=p_j + \omega$. For $t=p_1$ we define $\mathcal{R}(t)=L_1(0)$.
This function has the property that it approximates $f$ from finitely many noisy samples.
Figure \ref{fig:start_example} illustrates the reconstruction performance of the numerical reconstruction routine $\mathcal{R}$ applied to finitely many noisy spectrogram samples of a complex-valued signal. The following theorem proves stability and convergence of this algorithm.

\begin{theorem}\label{thmC}
Suppose that $f \in V_\beta^\infty(\varphi)$ has defining sequence $c \in \ell^\infty(\Z)$ and that there exist $$-s \eqqcolon p_1 < p_2 < \cdots < p_J \coloneqq s$$ with $|f(p_j)| \geq \gamma$ and $\max_{1 \leq j \leq J-1} (p_{j+1}-p_j)\leq r$. Let $I = [-s,s]$ and 
let $0 < \varepsilon \leq \min \left \{ \frac{\gamma^2}{2\sqrt{8}}, \frac{\gamma^3}{4 \| f \|_{L^\infty(I)}} \right \}$ and assume that the parameters $h,H,N$ of the grid $X$ have the property that $$\frac{1}{h}\gtrsim_{\sigma,\beta,r} \log\left ( \frac{J\| c \|_\infty^2}{\varepsilon} + 1 \right), \ N \gtrsim_{\sigma,\beta,r} s+\log\left(\frac{J\| c \|_\infty^2}{\varepsilon}\right),
$$
$$
\ H \gtrsim_{\sigma,\beta,r} \frac{1}{h} \left (\log\left(\frac{J\| c \|_\infty^2}{\varepsilon}\right)\right)^{1/2}.$$
If the noise level satisfies
$\| \eta \|_\infty \lesssim_{\sigma,\beta,r} \frac{\varepsilon}{hJ}$ then
\begin{equation}\label{epsbound}
    \min_{\tau \in \T} \| f - \tau \mathcal{R} \|_{L^\infty(I)} \leq 32\frac{\max \{ 1, \| f \|^2_{L^\infty(I)} \}}{\min \{ \gamma, \gamma^5 \}} (\varepsilon + \varepsilon^2).
\end{equation}
\end{theorem}

The numerical approximation routine $\mathcal{R}$ combined with Theorem \ref{thmC} results in a provably stable and convergent approximation algorithm as soon as condition \textbf{(P)} is satisfied. Note that this is an assumption on $|f|$ which is a-priori unknown if only samples of $|\mathcal{G}f|$ are available. However, this information is encoded in the spectrogram of $f$: we will show that the modulus $|f|$ of any $f \in V_\beta^\infty(\varphi)$ can be approximated in a uniformly and \emph{globally stable} way without imposing any additional assumption on $f$, see the discussion following Corollary \ref{cor:tensor_stability}, as well as Lemma \ref{lma:local_stability}. Based on a reconstruction of $|f|$ the resulting algorithm presented in Section \ref{sec:algo} first detects points $p_1, \dots, p_J$ such that condition \textbf{(P)} is satisfied and finally uses the numerical approximation routine $\mathcal{R}$ to provide a reconstruction of $f$. In this way, the information on the location of points $p_j$ can be omitted.

Finally, we study the growth behavior of the number of spectrogram samples needed to approximate a function on an interval $[-s,s]$. Assume that $f$ has the property that there exists a partition $\cdots < p_{-1} < p_0 < p_1 < \cdots$ of the real line such that $|f(p_j)| \geq \gamma$ for some $\gamma>0$ and all $j \in \Z$. Suppose we aim to approximate $f$ on the interval $[-s,s]$ using the numerical approximation routine $\mathcal{R}=\mathcal{R}(s)$. How does the minimal number of spectrogram samples $\mathcal{N}(s)$ grow so that $\mathcal{R}(s)$ achieves the bound \eqref{epsbound} for a fixed $\varepsilon>0$?

\begin{theorem}\label{thmD}
Let $\cdots < p_{-1} < p_0 < p_1 < \cdots$ be a partition of the real line such that $\sup_{j \in \Z} (p_{j+1}-p_j) < \infty$.
Let $\gamma > 0$ and $f \in V_\beta^\infty(\varphi)$ such that $|f(p_j)| \geq \gamma$ for all $j \in \Z$. Assume that the spectrogram samples $\mathfrak S$ are noiseless, $\eta = 0$, and $\varepsilon$ is given as in Theorem \ref{thmC}. 
Then the minimal number of spectrogram samples $\mathcal{N}(s)$ needed to achieve the bound \eqref{epsbound} on the interval $[-s,s]$ using the numerical approximation routine $\mathcal{R}$ satisfies
$$
\mathcal{N}(s) \lesssim_{\sigma,\beta} \log \left ( \frac{s\| c \|_\infty^2}{\varepsilon} \right )^{\frac{3}{2}} \left ( s + \log \left ( \frac{s\| c \|_\infty^2}{\varepsilon} \right ) \right ). 
$$
\end{theorem}

Theorems \ref{thmA:explicit_reconstruction}, \ref{thmB}, \ref{thmC} and \ref{thmD} are a direct consequence of Theorems \ref{thm:explicit_reconstruction}, \ref{thm:main_stability}, \ref{thm:qualitative} and Corollary \ref{cor:upper_growth_estimate} that will be proved in Sections \ref{sec:uaer}, \ref{sec:stab}, \ref{sec:atgss} and \ref{sec:atgss}, respectively.

\begin{figure}
\centering
\hspace*{-1.7cm}
   \includegraphics[width=15.8cm]{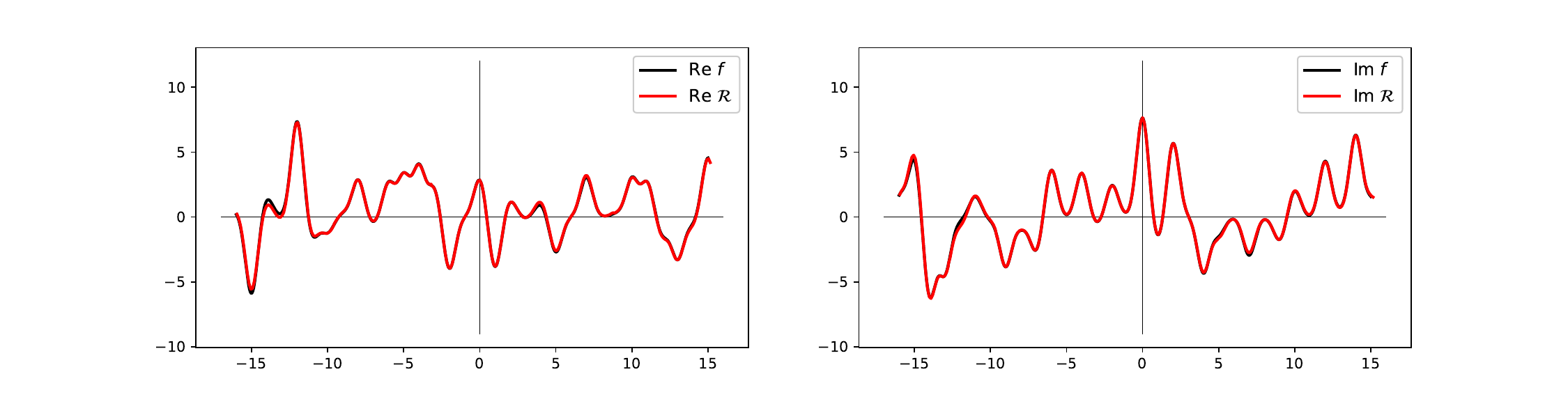}
\caption{Reconstruction of a complex-valued function $f \in V_\beta^\infty(\varphi)$ ($\beta = 1, \varphi(t)=e^{-\pi t^2}$) on the interval $[-16,16]$ using the numerical approximation routine $\mathcal{R}$ with spectrogram samples located at $X= \tfrac{1}{2} \{ -40, \dots, 40 \} \times \tfrac{1}{12}\{ -60, \dots, 60 \}$. To each spectrogram sample, Gaussian noise with mean zero and standard deviation 0.001 is added, i.e. the measurement matrix $\mathfrak{S}$ takes the form $\mathfrak{S} = |\mathcal{G}f(X)|^2+\mathcal{N}(0,\sigma^2), \sigma=0.001$.}
\label{fig:start_example}
\end{figure}

\subsection{Outline}

The article is structured as follows: in Section \ref{section2} we present the necessary general theory on shift-invariant spaces which will be used throughout the article. Section \ref{section3} starts with an abstract discussion on the reconstruction of arbitrary maps from their tensor product. These ideas will be applied to Gaussian shift-invariant spaces and lead to an explicit inversion formula from spectrogram measurements which makes use of a biorthogonal expansion. The section ends with a stability analysis of the derived reconstruction formula. In the final section, Section \ref{sec:algo}, we present and analyze an algorithmic way of approximating functions in Gaussian shift-invariant spaces from finitely many spectrogram samples. Moreover, we elaborate that the obtained algorithm approximates signals in Paley-Wiener spaces.

\section{Preliminaries on shift-invariant spaces}\label{section2}

\subsection{General theory}

For $\beta>0, p \in [1,\infty]$ and $\phi \in L^p(\R)$ let the shift-invariant space $V_\beta^p(\phi)$ be defined as in \eqref{def:si_space}. Recall that a sequence $(f_n)_n \subset \mathcal{H}$ in a separable Hilbert space $\mathcal{H}$ is said to be a Riesz basis for $\mathcal{H}$ if it is complete in $\mathcal{H}$ and there exist positive constants $0<A\leq B$ such that for arbitrary $n \in \N$ and arbitrary $c_1,\dots, c_n \in \C$ one has
\begin{equation}
    A \sum_{j=1}^n |c_j|^2 \leq \left \| \sum_{j=1}^n c_j f_j \right \|^2 \leq B \sum_{j=1}^n |c_j|^2.
\end{equation}
The optimal constants $A$ and $B$ are called the lower and upper bound of $(f_n)_n$. Every Riesz basis with lower and upper bounds $A$ and $B$ is a frame with frame constants $A$ and $B$. Associate to a Riesz basis $(f_n)_n$ its frame operator $S:\mathcal{H} \to \mathcal{H}$, defined by
$
Sf = \sum_{n \in \Z} \langle f,f_n \rangle f_n.
$
Then $S$ is a positive, invertible operator and $(S^{-1}f_n)_n$ is biorthogonal to $(f_n)_n$, i.e.
$$
\langle f_n, S^{-1}f_k \rangle = \delta_{nk} \coloneqq \begin{cases} 
      1 & n=k \\
      0 & n \neq k
   \end{cases}.
$$
Now let $\mathcal{H} = V_\beta^2(\phi)$ be a shift-invariant space with a square-integrable generator $\phi$. Define the 1-periodic map $\Phi_\beta : \R \to \R$ via
\begin{equation}\label{eq:Phi}
    \Phi_\beta(t) = \sum_{n\in\Z} | \hat \phi (\tfrac{t+n}{\beta})|^2
\end{equation}
where $\hat{\phi} = \ft \phi$ denotes the Fourier transform of $\phi$ which is defined on $L^1(\R) \cap \lt$ by
$$
\hat{\phi}(\omega) = \ft \phi(\omega) = \int_\R \phi(t) e^{-2\pi i \omega t} \, dt
$$
and extends to a unitary operator on $\lt$. The map $\Phi_\beta$ characterizes the property of $(T_{\beta n} \phi)_n$ being a Riesz basis for $V_\beta^2(\phi)$ \cite[Theorem 9.2.5]{christensenBook}.

\begin{theorem}\label{thm:riesz_basis_characterization}
Let $\phi \in \lt$ and $\beta>0$. Suppose that there exists $0<A\leq B < \infty$ such that
\begin{equation}\label{eq:riesz_basis_characterization}
    \beta A \leq \Phi_\beta(t) \leq \beta B
\end{equation}
for almost every $t \in [0,1]$. Then $V_\beta^2(\phi)$ is a closed subspace of $L^2(\R)$ and $(T_{\beta n}\phi)_n$ is a Riesz basis for $V_\beta^2(\phi)$ with bounds $A$ and $B$.
\end{theorem}

If $\phi$ satisfies condition \eqref{eq:riesz_basis_characterization} and if $S : V_\beta^2(\phi) \to V_\beta^2(\phi)$ is the frame operator associated with $(T_{\beta n} \phi)_n$ then $S$ commutes with $\beta\Z$-shifts, implying that $(S^{-1}T_{\beta n}\phi)_n = (T_{\beta n}\tilde \phi)_n$ with $\tilde \phi = S^{-1}\phi$ \cite[Lemma 9.4.1]{christensenBook}. Accordingly, the canonical dual frame $(S^{-1}T_{\beta n} \phi)_n$ has the same structure as $(T_{\beta n} \phi)_n$.

\begin{definition}
Let $\beta >0$ and $\phi \in \lt$ such that $(T_{\beta n} \phi)_n$ is a Riesz basis for $V_\beta^2(\phi)$. If $S$ is the frame operator associated with $(T_{\beta n} \phi)_n$ then $\widetilde \phi \coloneqq S^{-1}\phi $ is called the dual generator of $\phi$.
\end{definition}

Clearly, the dual generator generates the same shift-invariant space, i.e. $V_\beta^2(\phi) = V_\beta^2(\widetilde \phi)$. In the situation of Theorem \ref{thm:riesz_basis_characterization}, the map $\widetilde \phi$ can be derived explicitly. To do so, define in a similar fashion as above the $\frac{1}{\beta}$-periodization of $|\hat \phi|^2$ as
\begin{equation*}
    \Psi_\beta(t) = \sum_{n \in \Z} | \hat \phi (t+\tfrac{n}{\beta}) |^2.
\end{equation*}
The dual generator is then given as a formula of $\hat \phi$ and $\Psi$.

\begin{theorem}\label{thm:dual_generator}
Suppose that $(T_{\beta n} \phi)_n$ is a Riesz basis for $V_\beta^2(\phi)$ with frame operator $S$. Let $D = \{ t \in \R : \Psi_\beta(t) \neq 0 \}$ and let $\bone_D$ be the indicator function of $D$. If the map $\theta$ is defined as
$$
\theta (\omega) = \frac{\beta \hat \phi (\omega)}{\Psi_\beta(\omega)} \bone_D(\omega)
$$
then $\tilde \phi \coloneqq \ift \theta$ is the dual generator of $\phi$.
\end{theorem}
\begin{proof}
This follows from a generalization of \cite[Proposition 9.4.2]{christensenBook}. For the convenience of the reader, we provide a proof in the Appendix \ref{appendix:A}.
\end{proof}

Note that the structure of $V_\beta^2(\phi)$ can be enriched by imposing further properties on the generator $\phi$. For instance, if $\phi$ is continuous and belongs to the Wiener Amalgam space then $V_\beta^2(\phi)$ is a reproducing kernel Hilbert space (RKHS). We refer to \cite{GroechenigSurvey,christensenBook,CHRISTENSEN200448} for results in this direction.

\subsection{Gaussian generators}\label{subsection:gaussian_generators}

We now specialize $\phi$ to be the Gaussian $\varphi^\sigma$ with variance $\sigma^2$ as defined in Equation \eqref{def:gaussian}. For the sake of simplicity we write $\varphi = \varphi^\sigma$ with the implicit understanding that $\varphi$ has variance $\sigma^2$. Properties of $V_\beta^2(\varphi)$ can be expressed in terms of the Jacobi theta function of third kind. Recall that this is the map $\vartheta_3 : \C \times (0,1) \to \C$,
$$
\vartheta_3(z,c) = \sum_{n \in \Z} c^{n^2}e^{2niz}.
$$
For every fixed $c \in (0,1)$, $\vartheta_3$ is an entire, $\pi$-periodic map in $z$. Using the results of the previous subsection gives

\begin{corollary}\label{cor:gaussian_sispace}
Let $\varphi$ be a centred Gaussian with variance $\sigma^2$. Then the $1$-periodization of $\varphi$ is given by
$$
\Phi_\beta(t) = \sigma \beta \sqrt \pi \vartheta_3(\pi t, \varphi(\tfrac{\beta}{\sqrt 2})).
$$
In particular, the system $(T_{\beta n}\varphi)_n$ constitutes a Riesz basis for $V_\beta^2(\varphi)$. Moreover, if the function $\Lambda : \R \to \C$ is defined as
$$
\Lambda(t) \coloneqq \frac{e^{-2\pi^2\sigma^2t^2}}{\vartheta_3(\pi\beta t, \varphi(\frac{\beta}{\sqrt{2}}))}
$$
then the dual-generator of $\varphi$ is given by $\widetilde \varphi = \sqrt{2} \ift \Lambda$.
\end{corollary}
\begin{proof}
The identity $\ft \varphi(t) = \sigma \sqrt{2\pi} e^{-2\pi^2\sigma^2 t^2}$ implies that the 1-periodization $\Phi_\beta$ is given by
\begin{equation}
\begin{split}
\Phi_\beta(t) & =  2\pi \sigma ^2 \sum_n e^{-\frac{4\pi^2\sigma^2}{\beta^2}(t+n)^2} = \sigma \beta \sqrt \pi \sum_n e^{-\frac{\beta^2 n^2}{4 \sigma^2}} e^{2 \pi i n t} \\
& = \sigma \beta \sqrt \pi \vartheta_3(\pi t, \varphi(\tfrac{\beta}{\sqrt 2})),
\end{split}
\end{equation}
where the second equality follows from Poisson's summation formula. In a similar fashion, we obtain
\begin{equation}
\Psi_\beta(t) = \sqrt{\pi} \sigma \beta \vartheta_3(\pi \beta t, \varphi(\tfrac{\beta}{\sqrt{2}})),
\end{equation}
thereby proving the statement using Theorem \ref{thm:dual_generator}.
\end{proof}

\begin{figure}[b]\label{fig:dual_window}
\centering
\hspace*{-1.4cm}
   \includegraphics[width=15cm]{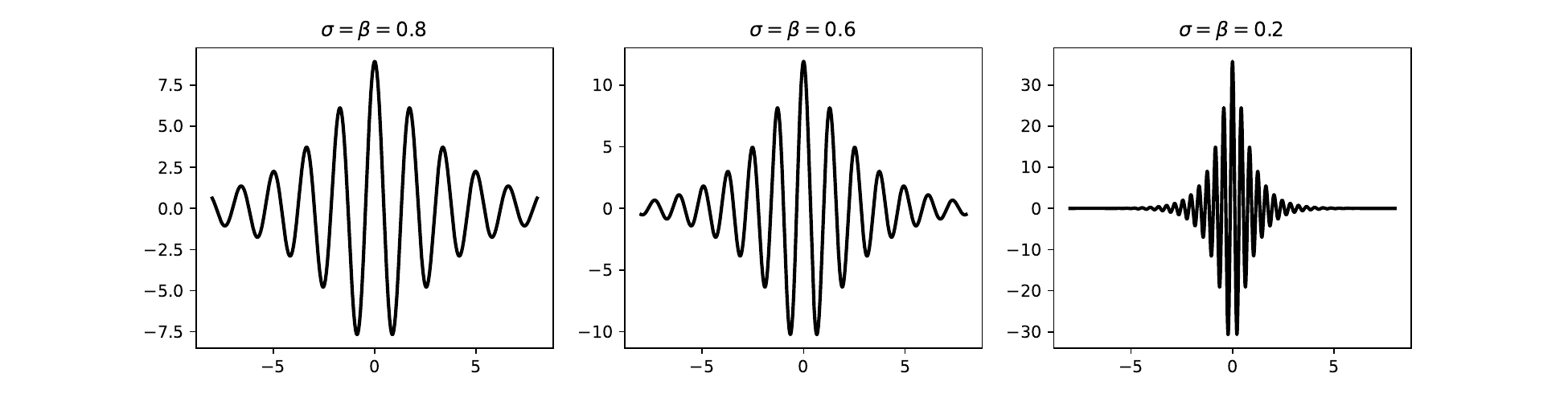}
\caption{Dual generators of Gaussians $\varphi^\sigma(t)$ associated with the shift-invariant space $V_\beta^2(\varphi)$ with step-sizes $\beta = 0.8, 0.6, 0.2$ and $\sigma=\beta$.}
\end{figure}

Note that the zeros of $\vartheta_3$ are bounded away from the real axis. Hence, the reciprocal theta function $\frac{1}{\vartheta_3}$ defines an analytic and periodic function in an open strip containing the real axis. It follows that the Fourier coefficients $a_n = a_n(c)$ of $\frac{1}{\vartheta_3}$ decay at least exponentially. If
\begin{equation}\label{theta_identity_1}
    \frac{1}{\vartheta_3(\pi \beta t,c)} = \sum_{n \in \Z} a_n e^{2 \pi i \beta n t }
\end{equation}
with $c = \varphi(\tfrac{\beta}{\sqrt{2}})$ then the identity $\ft(s \mapsto e^{-2\pi^2\sigma^2s^2})(t) = (\sigma\sqrt{2\pi})^{-1}\varphi(t)$ implies that $\widetilde \varphi$ can be written as
$$
\widetilde \varphi (t) = \frac{1}{\sigma \sqrt{\pi}} \sum_n a_n T_{-\beta n} \varphi(t).
$$
Hence, up to a multiplicative constant, the defining sequence of the dual generator is given by the Fourier coefficients of the reciprocal theta function and therefore the decay of $\widetilde \varphi$ matches with the decay of $(a_n)_n$. In fact, the Fourier coefficients of $\frac{1}{\vartheta_3}$ satisfy precisely an exponential decay (see a derivation by Janssen \cite[p. 178]{JANSSEN1996165} based on results of Whittaker and Watson \cite[p. 489]{whittaker_watson_1996}): if the constant $\xi = \xi(c)$ is defined by
\begin{equation}\label{theta_identity_2}
    \xi = \sum_{n \in \mathbb Z} (-1)^n (2n+1)c^{(n+\frac 1 2)^2}
\end{equation}
then
\begin{equation}\label{theta_identity_3}
    a_n = (-1)^n\frac{2}{\xi} \sum_{m=0}^\infty (-1)^m c^{(m+\frac 1 2)(2|n|+m+\frac 1 2)}.
\end{equation}
This explicit formula for the Fourier coefficients will be used in the upcoming sections to derive decay estimates on the dual generator leading to stability and approximation results of the Gabor phase retrieval problem. We conclude the present section with the observation that if $f$ has a Gaussian shift-invariant structure and is merely assumed to be bounded then the inner products $\langle f, T_{\beta n} \varphi \rangle \coloneqq \int_\R f(t) T_{\beta n} \varphi(t) \, dt$ remain well-defined and the map
\begin{equation}\label{eq:expansion}
    t \mapsto \sum_n \langle f, T_{\beta n} \varphi \rangle T_{\beta n} \widetilde \varphi(t)
\end{equation}
defines a smooth function. If $f \in V_\beta^2(\varphi)$ then the series in \eqref{eq:expansion} converges both in $\lt$ and in $L^\infty(\R)$ to $f$. If $f \in V_\beta^\infty(\varphi)$ then the following can be said.

\begin{proposition}\label{prop:V_2_V_inf}
If $f \in V_\beta^\infty(\varphi)$ then $f(t)=\sum_n \langle f, T_{\beta n} \varphi \rangle T_{\beta n} \widetilde \varphi(t)$ where the series on the right converges uniformly on compact intervals of the real line.
\end{proposition}
\begin{proof}
See the Appendix \ref{appendix:B}.
\end{proof}

We call $\sum_n \langle f, T_{\beta n} \varphi \rangle T_{\beta n} \widetilde \varphi$ the biorthogonal expansion of $f \in V_\beta^\infty(\varphi)$. Note that $V_\beta^\infty(\varphi)$ consists of smooth functions and is the largest among the spaces $V_\beta^p(\varphi)$ due to the inclusion $V_\beta^p(\varphi) \subset V_\beta^q(\varphi) \subset C^\infty(\R), 1 \leq p \leq q \leq \infty$ \cite[Corollary 2.5]{GroechenigSurvey}.

\section{Phase retrieval in Gaussian shift-invariant spaces}\label{section3}

\subsection{Uniqueness and explicit reconstruction}\label{sec:uaer}

Let $g \in L^2(\R)$ be a window function, $\mathcal{V}_gf$ the short-time Fourier transform of $f \in L^\infty(\R)$ and $\mathcal{G}f$ the Gabor transform of $f$.
We aim to recover $f$ from $|\mathcal{V}_gf(X)|$ where $X \subseteq \R^2$. Clearly, this will only be possible up to a global phase. In order to decide whether two functions agree up to global phase one can study their tensor product $f_\omega$ as defined in \eqref{def:tensor}.

\begin{proposition}\label{prop:tensorproduct_properties}
Let $f,h : \R \to \C$ be two complex-valued maps and denote by $F,H : \R^2 \to \C$ the functions $F(t,\omega) \coloneqq f_\omega(t), H(t,\omega) = h_\omega(t)$.
\begin{enumerate}
\setlength\itemsep{0em}
\item If $p \in \R$ such that $F(p,0) \neq 0$  then for every $\omega \in \R$ we have $$f(p+\omega)= \tau F(p,0)^{-\frac{1}{2}}\overline{F(p+\omega,\omega)}, \ \tau = \frac{f(p)}{|f(p)|}.$$
\item If $p \in \R$ such that $F(p,0) \neq 0$ and $H$ agrees with $F$ on the diagonal segment $D^r = \{ (p+\omega, \omega) : 0 \leq \omega \leq r \}, r \in [0,\infty),$ then there exists a $\tau \in \T$ such that $f(p+\omega)=\tau h(p+\omega)$ for every $|\omega| \leq r$.
\item If $F=H$ on $\R^2$ then there exists a  $\tau \in \T$ such that $f=\tau h$ on the entire real line.
\item Suppose that $p_1 < p_2 < \cdots < p_J \in \R$ such that $F(p_j,0) \neq 0$. For $j = 1, \dots, J-1$ define function $L_j$ and phases $\nu_j$ by
$$
L_j(\omega) = F(p_j,0)^{-\frac{1}{2}}\overline{F(p_j+\omega,\omega)}, \ \nu_j = \frac{L_j(p_{j+1}-p_j)}{|L(p_{j+1}-p_j)|}, \ \nu_0=1.
$$
Then there exists a $\tau \in \T$ with the following property: Whenever $t \in [p_1,p_J]$ with $t = p_j + \omega$ where $p_j$ is the largest value not exceeding $t$ then
\begin{equation}\label{eq:product_elementary}
f(t) = \tau \nu_1 \nu_2 \cdots \nu_{j-1} L_j(\omega).
\end{equation}
In particular, $f$ can be reconstructed on the interval $[p_1,p_J]$ from the values $F(X)$ where $X$ is the union of the diagonal line segments $D_j = \{ (p_j+\omega, \omega) : 0\leq \omega \leq p_{j+1}-p_j \}, j= 1,\dots, J-1$.
\end{enumerate}
\end{proposition}
\begin{proof}
Observe that $F(p,0)=f_0(p)=|f(p)|^2$. The first statement then follows directly from the equations
$$
F(p+\omega,\omega) = f_\omega(p+\omega) = f(p)\overline{f(p+\omega)} = \frac{|f(p)|}{\overline{f(p)}} F(p,0)^{\frac{1}{2}} \overline{f(p+\omega)}.
$$
The second statement follows directly from the first and the third follows from the second. It remains to show Part 4. Let $L_j$ and $\nu_j$ be defined as above. We have by definition
$$
L_j(p_{j+1}-p_j) = \underbrace{F(p_j,0)^{-\frac{1}{2}}}_{>0} \overline{f(p_j)} f(p_{j+1}).
$$
Therefore, $\nu_j = \frac{f(p_{j+1})|f(p_{j})|}{f(p_{j})|f(p_{j+1})|}$ which gives $$\nu_1 \cdots \nu_{j-1} = \frac{|f(p_1)|f(p_j)}{f(p_1)|f(p_j)|}.$$ Thus,
$$
\tau \nu_1 \cdots \nu_{j-1} L_j(\omega) = f(p_j+\omega)
$$
with $\tau = \frac{f(p_1)}{|f(p_1)|}$ which further shows that $f$ can be reconstructed on $[p_1,p_J]$ up a global phase from the union of the diagonal line segments $D_j$.
\end{proof}

Part 4 of the preceding proposition can be interpreted as follows: if the tensor product is known on the diagonal line segments $D_j, 1 \leq j \leq J-1$, then $f$ can be reconstructed on the interval $[p_j,p_{j+1}]$ up to a global phase and this reconstruction is given by $L_j$. If on each such interval the local reconstruction is multiplied by the phase $\tau \nu_1 \nu_2 \cdots \nu_{j-1}$ then the resulting function defined on the union of all the intervals $[p_j,p_{j+1}]$ agrees up to the phase factor $\tau$ with $f$. We call the multiplication by $\nu_1 \nu_2 \cdots \nu_{j-1}$ a phase synchronization. The phase synchronization is possible whenever $f(p_j) \neq 0$. This motivates

\begin{definition}
Let $p_1,\dots,p_J \in \R$ and $\gamma > 0$. We say that $f: \R \to \C$ satisfies condition \textbf{(P)} if
\[
J \geq 2 \ \ \land \ \ p_1 < p_2 < \cdots < p_J \ \ \land \ \ |f(p_j)| \geq \gamma \ \forall j \in \{ 1, \dots, J \} \tag{\textbf{P}}.
\]
\end{definition}

Suppose now that $f$ is an element of the Gaussian shift-invariant space $V_\beta^\infty(\varphi)$ where $\beta > 0$ is an arbitrary step-size. An important property of $V_\beta^\infty(\varphi)$ is that the shift-invariant structure is invariant under the tensor product operation.

\begin{proposition}\label{prop:invariance_tensor_product}
Let $\varphi=\varphi^\sigma$ be a Gaussian window function with variance $\sigma^2$ and let $V_\beta^\infty(\varphi)$ be the corresponding shift-invariant space with generator $\varphi$ and step-size $\beta >0$. If $f \in V_\beta^\infty(\varphi)$ then $f_\omega \in V_{\frac{\beta}{2}}^\infty(\varphi_\omega)$ for every $\omega \in \R$.
\end{proposition}
\begin{proof}
Let $c = (c_n)_n \in \ell^\infty(\Z)$ be the defining sequence of $f$, i.e. $f = \sum_n c_n T_{\beta n}\varphi \in V_\beta^\infty(\varphi)$. The tensor product of $f$ is given by
$$
f_\omega = (T_\omega f)\overline{f} = \sum_n \sum_k c_n \overline{c_k} T_{\beta n+\omega} \varphi T_{\beta k} \varphi.
$$
Using the product formula
$$
\varphi(t-a)\varphi(t-b) = e^{-\frac{(a-b)^2}{4\sigma^2}}\varphi^{\frac{\sigma}{\sqrt{2}}}(t-\tfrac{a+b}{2}), \ \ a,b \in \R, \ \ \sigma>0
$$
we obtain for every $n,k \in \Z$ the identity
\begin{equation*}
\begin{split}
T_{\beta n+\omega} \varphi T_{\beta k} \varphi & = e^{-\frac{(\beta(n-k)+\omega)^2}{4 \sigma^2}} \varphi^{\tfrac{\sigma}{\sqrt{2}}}\left(t- \frac{\beta(n+k)}{2} - \frac{\omega}{2}\right) \\
& = e^{-\frac{(\beta(n-k)+\omega)^2}{4 \sigma^2}} e^{\frac{\omega^2}{4 \sigma^2}} \varphi\left(t-\frac{\beta(n+k)}{2}-\omega \right ) \varphi\left ( t-\frac{\beta(n+k)}{2} \right) \\
& = e^{-\frac{(\beta(n-k)+\omega)^2}{4 \sigma^2}} e^{\frac{\omega^2}{4 \sigma^2}} T_{\frac{\beta}{2}(n+k)} \varphi_\omega.
\end{split}
\end{equation*}
Setting $A(n,k) \coloneqq e^{-\frac{(\beta(n-k)+\omega)^2}{4 \sigma^2}} e^{\frac{\omega^2}{4 \sigma^2}}$ and defining coefficients $(d_\ell)_\ell \subset \C$ via
$$
d_\ell \coloneqq \sum_{n,k \in \Z, n+k=\ell} A(n,k) c_n \overline{c_k}
$$
shows that $f_\omega$ takes the form
$$
f_\omega = \sum_{\ell} d_\ell T_{\frac{\beta}{2}\ell} \varphi_\omega.
$$

The statement follows at once from the bound $|d_\ell| \leq \| c \|_\infty^2 \sum_n A(n,\ell-n)$ and the fact that the Gaussian sum $\sum_n A(n,\ell-n)$ is uniformly bounded in $\ell$.
\end{proof}

We aim to reconstruct $f$ from its tensor-product $f_\omega$ using the abstract reconstruction result given in Proposition \ref{prop:tensorproduct_properties} (4). Since $f_\omega$ is generated by $\varphi_\omega$ one could recover $f_\omega$ from the inner products $\langle f_\omega, T_{\frac{\beta}{2}n} \varphi_\omega \rangle, n \in \Z,$ using a biorthogonal expansion as discussed in Section \ref{subsection:gaussian_generators}. We already verified the Riesz basis property of $(T_{\frac{\beta}{2}n} \varphi_\omega)_n$ in Corollary \ref{cor:gaussian_sispace}. It remains open how to gain access to the inner products $\langle f_\omega, T_{\frac{\beta}{2}n} \varphi_\omega \rangle \in \C$. This information is encoded in the spectrogram of $f$.

\begin{proposition}\label{prop:fourier_identity}
For every $f \in L^\infty(\R)$ and every $g \in \lt$ one has
$$
\ft |\mathcal{V}_gf(x,\cdot)|^2(\omega) = \langle f_\omega, T_xg_\omega \rangle.
$$
\end{proposition}
\begin{proof}
Under the above assumptions we have $f\overline{T_xg} \in \lt$ and therefore $|\ft(f\overline{T_xg})|^2 \in L^1(\R)$. Consequently, the Fourier transform of $t \mapsto |\mathcal{V}_gf(x,t)|^2 = |\ft(f\overline{T_xg})(t)|^2$ is pointwise defined (as a continuous function).
Now let $p_x(t) \coloneqq f(t+x/2)\overline{g(t-x/2)}$. From the definition of the cross-ambiguity function $A(f,g)=e^{\pi i x \omega}\mathcal{V}_gf(x,\omega)$ it follows that
$$
A(f,g)(x,\omega)\overline{A(f,g)(x,\omega)} = \ft p_x(\omega) \overline{\ft p_x (\omega)} = \ift(\ft^2 p_x * \overline{p_x})(\omega).
$$
The convolution $\ft^2p_x * \overline{p_x}$ can be written as
$$
(\ft^2p_x * \overline{p_x})(\omega) = \int_\R T_\omega f(k) \, \overline{f(k)} \, T_x[\overline{T_\omega g(k)} \, g(k)] \, dk = \langle f_\omega, T_x g_\omega \rangle,
$$
where we used the change of variables $k \mapsto k - \frac{x}{2}$ in the first equality. It follows that
$$
|\mathcal{V}_gf(x,\omega)|^2 = A(f,g)(x,\omega)\overline{A(f,g)(x,\omega)} = \ift(t \mapsto \langle f_t, T_x g_t \rangle)(\omega)
$$
which yields the assertion.
\end{proof}

Denoting the dual-generator of $\varphi_\omega$ by $\widetilde{\varphi_\omega}$, the following result is an immediate consequence of Proposition \ref{prop:fourier_identity}.

\begin{corollary}\label{cor:tp_expansion}
Let $f \in V_\beta^\infty(\varphi)$. If
$$
f_\omega = \sum_{n \in \Z} c_n T_{\frac{\beta}{2}n}\widetilde{\varphi_\omega}, \ (c_n)_n \in \ell^\infty(\Z)
$$
is the biorthogonal expansion of $f_\omega$ w.r.t. the dual generator $\widetilde{\varphi_\omega}$ then
\begin{equation}\label{cn_coefficients}
    c_n = \int_\R |\mathcal{G}f(\tfrac{\beta}{2}n,t)|^2e^{-2\pi i \omega t} \, dt
\end{equation}
for every $n \in \Z$. In particular, every $f_\omega$ is determined uniquely by $|\mathcal{G}f(\tfrac{\beta}{2}\Z \times \R)|$ and every $f  \in V_\beta^\infty(\varphi)$ is determined up to a global phase by $|\mathcal{G}f(\tfrac{\beta}{2}\Z \times \R)|$, i.e. if $f,h \in V_\beta^\infty(\varphi)$ with $|\mathcal{G}f(\tfrac{\beta}{2}\Z \times \R)|=|\mathcal{G}h(\tfrac{\beta}{2}\Z \times \R)|$ then there exists a $\tau \in \T$ such that $f = \tau h$.
\end{corollary}
\begin{proof}
The expansion $f_\omega = \sum_{n \in \Z} c_n T_{\frac{\beta}{2}n}\widetilde{\varphi_\omega}$ with $(c_n)_n$ given as in \eqref{cn_coefficients} is a direct consequence of Proposition \ref{prop:invariance_tensor_product} and Proposition \ref{prop:fourier_identity} above. In particular, this implies that $f_\omega$ is determined uniquely by $|\mathcal{G}f(\tfrac{\beta}{2}\Z \times \R)|$. Since $\omega$ was arbitrary, Proposition \ref{prop:tensorproduct_properties} shows that $f$ is determined up to a global phase factor by $|\mathcal{G}f(\tfrac{\beta}{2}\Z \times \R)|$.
\end{proof}

Observe that if $f$ takes the form $f=\sum_n c_n T_{\beta n}\varphi$ with $(c_n)_n \subset \R$, i.e. $f$ is real-valued, then the Gabor phase retrieval problem can be interpreted as a reconstruction of $f$ from its modulus (up to a variance-change in the generator). For if the frequency variable in the Gabor transform is set to zero then
\begin{equation}\label{frequency_zero}
    |\mathcal{G}f(x,0)| = \left | \sum_n c_n (\varphi * T_{\beta n} \varphi)(x) \right | = \left | \sum_n \sigma \sqrt \pi c_n T_{\beta n} \varphi^{\sqrt{2}\sigma}(x) \right | =: |h(x)|.
\end{equation}
The map $h$ is an element of $V_\beta^\infty(\varphi^{\sqrt{2}\sigma})$ and a result due to Gröchenig shows that $h$ is determined by $|h(X)|$ whenever $X \subseteq \R$ has lower Beurling density $D^-(X)>2\beta$ \cite[Theorem 1]{groechenigPhase}. In particular, $f$ is determined by $|\mathcal{G}f(\frac{\beta}{2+\varepsilon}\Z \times \{ 0 \})|$. This result is sharp and does not hold for complex-valued maps \cite[Section 2]{groechenigPhase}. Corollary \ref{cor:tp_expansion} shows that in the complex case, $(c_n)_n \subset \C$, uniqueness can be guaranteed if we set $\varepsilon=0$ and extend the grid $\frac{\beta}{2}\Z \times \{ 0 \}$ to parallel lines $\frac{\beta}{2}\Z \times \R$ in the time-frequency plane. In order to transform the uniqueness part of Corollary \ref{cor:tp_expansion} into a uniqueness result from measurements lying on a lattice, one may suppose further properties on the step-size $\beta$ as in \cite[Theorem 3.6]{grohsliehr2020}. We conclude the present section with an explicit reconstruction formula from spectrogram measurements.

\begin{theorem}\label{thm:explicit_reconstruction}
Let $f \in V_\beta^\infty(\varphi)$ and let $p \in \R$ such that $f(p) \neq 0$. Then there exists a unimodular constant $\tau \in \T$ such that
$$
f(p+\omega) = \tau |f_0(p)|^{-\frac{1}{2}} \sum_{n \in \Z} \left ( \int_\R |\mathcal{G}f(\tfrac{\beta}{2}n,t)|^2e^{2\pi i \omega t} \, dt \right ) T_{\frac{\beta}{2}n} \widetilde{\varphi_\omega}(p+\omega)
$$
for every $\omega \in \R$ and
$$
f_0(p) = \sum_{n \in \Z} \left ( \int_\R |\mathcal{G}f(\tfrac{\beta}{2}n,t)|^2 \, dt \right ) T_{\frac{\beta}{2}n} \widetilde{\varphi_0}(p)
$$
\end{theorem}
\begin{proof}
Combining Corollary \ref{cor:tp_expansion} with Proposition \ref{prop:tensorproduct_properties} and fact that the dual generator $\widetilde{\varphi_\omega}$ is real-valued yields the statement.
\end{proof}

Theorem \ref{thm:explicit_reconstruction} implies that in a Gaussian shift-invariant setting it suffices to know the spectrogram on parallel vertical lines in order to reconstruct $f$ from its tensor product. In our further analysis it will be useful to work with an explicit formula for the dual generator of the tensor product $\varphi_\omega$. Such a formula can be readily established. The product formula
$$
\varphi(t-a)\varphi(t-b) = e^{-\frac{(a-b)^2}{4\sigma^2}}\varphi^{\frac{\sigma}{\sqrt{2}}}(t-\tfrac{a+b}{2}), \ \ a,b \in \R, \ \ \sigma>0
$$
implies that $\varphi_\omega(t) = e^{-\frac{\omega^2}{4\sigma^2}} T_{\frac{\omega}{2}}\varphi^{\frac{\sigma}{\sqrt 2}}(t)$. Using Corollary \ref{cor:gaussian_sispace} and substituting $\sigma$ with $\frac{\sigma}{\sqrt{2}}$ and $\beta$ with $\frac{\beta}{2}$ gives
\begin{equation}\label{eq:Lambda}
\widetilde{\varphi_\omega} = \sqrt{2}e^{-\frac{\omega^2}{4\sigma^2}}T_{\frac{\omega}{2}} \ift \Lambda, \ \ \Lambda(t) = \frac{e^{-\pi^2 \sigma^2 t^2}}{\vartheta_3(\frac{\beta}{2}\pi t, \varphi(\frac{\beta}{2}))}.
\end{equation}

\subsection{Stability}\label{sec:stab}

This section is devoted to the stability analysis of the biorthogonal expansion of the tensor product $f_\omega$ as well as the explicit inversion formula presented in Theorem \ref{thm:explicit_reconstruction}.
Let the mixed norm $\| \cdot \|_{\alpha,p}$ be defined as in equation \eqref{def:mixed_norm}. Using this norm, stability estimates for tensor products of functions in $V_\beta^1(\varphi)$ can be derived.

\begin{corollary}\label{cor:tensor_stability}
Let $f,h \in V_\beta^1(\varphi), \omega \in \R$, and suppose that $A(\sigma,\beta,\omega)$ is the lower bound of the Riesz basis $(T_{\frac{\beta}{2}n}\varphi_\omega)_n$ as given in Theorem \ref{cor:gaussian_sispace}. Then
\begin{equation}\label{ineq:stability_tensor_product}
    \| f_\omega - h_\omega \|_{\lt}^2 \leq \frac{1}{A(\sigma,\beta,\omega)} \| |\mathcal{G}f|^2-|\mathcal{G}h|^2 \|_{\frac{\beta}{2},2}^2 
\end{equation}
where $A(\sigma,\beta,\omega) = \frac{\sigma}{\sqrt 2} \varphi(\omega) \vartheta_3(\frac{\pi}{2},\varphi(\frac{\beta}{2}))$
\end{corollary}
\begin{proof}
By Hölder's inequality we have $f_\omega, h_\omega \in V_{\frac{\beta}{2}}^2(\varphi_\omega)$. The fact that $(T_{\frac{\beta}{2}n}\varphi_\omega)_n$ is a Riesz basis for $V_{\frac{\beta}{2}}^2(\varphi_\omega)$ implies that
$$
A(\sigma,\beta,\omega) \| f_\omega - h_\omega \|_{\lt}^2 \leq \sum_{n} | \langle f_\omega -h_\omega , T_{\frac{\beta}{2}n}\varphi_\omega \rangle |^2.
$$
In addition, Proposition \ref{prop:fourier_identity} shows that
\begin{equation*}
    \begin{split}
        | \langle f_\omega -h_\omega, T_{\frac{\beta}{2}n}\varphi_\omega \rangle | & = | \ft (|\mathcal{G}f(\tfrac{\beta}{2}n,\cdot)|^2 - |\mathcal{G}h(\tfrac{\beta}{2}n,\cdot)|^2)(\omega)| \\
        & \leq \| |\mathcal{G}f(\tfrac{\beta}{2}n,\cdot)|^2 - |\mathcal{G}h(\tfrac{\beta}{2}n,\cdot)|^2 \|_{L^1(\R)}.
    \end{split}
\end{equation*}
Combining the previous two inequalities yields the first part of the statement. To obtain the second assertion, we observe that the 1-periodization of $\varphi_\omega$ with step-size $\frac{\beta}{2}$ is given by
$
\Phi_{\beta/2}(t) = \frac{\sigma\beta}{2\sqrt 2} \varphi(\omega) \vartheta_3(\pi t , \varphi(\frac{\beta}{2})).
$
According to Theorem \ref{thm:riesz_basis_characterization}, the value $A(\sigma,\beta,\omega)$ is then given as the minimum of the map
$$t \mapsto \frac{2\Phi_{\beta/2}(t)}{\beta}$$ on $[0,1]$.
It was shown by Janssen \cite[p. 178]{JANSSEN1996165} that for every $c \in (0,1)$ the theta function $\vartheta_3(\cdot,c)$ attains its minima at $t \in \frac{\pi}{2}\Z$, thereby proving the statement.
\end{proof}

By virtue of the RKHS structure it is evident that an analogue version of Corollary \ref{cor:tensor_stability} can be derived where the $L^2$-norm on the left-hand side is replaced by the $L^\infty$-norm. For, if $f,h \in V_\beta^1(\varphi)$ then $f_\omega,h_\omega \in V^2_{\frac{\beta}{2}}(\varphi_\omega)$ and $V^2_{\frac{\beta}{2}}(\varphi_\omega)$ is a RKHS. Therefore, the point evaluation functionals
$$
f_\omega \mapsto f_\omega(x), \ \ x \in \R,
$$
are continuous linear functionals on $V^2_{\frac{\beta}{2}}(\varphi_\omega)$ and there exists $K_x \in V^2_{\frac{\beta}{2}}(\varphi_\omega)$ such that $f_\omega(x) = \langle f_\omega,K_x \rangle$ which implies that
$$
|f_\omega(x)-h_\omega(x)| \leq \| K_x \|_{L^2(\R)} \| f_\omega - h_\omega \|_{L^2(\R)}.
$$
Combining the fact that the generator is a Gaussian with an explicit form of $K_x$ (see, for instance, \cite[Theorem 4.1]{GroechenigSurvey}), we observe that $x \mapsto \| K_x \|_{L^2(\R)}$ is bounded on $\R$ which in turn yields the bound
$$
\| f_\omega - h_\omega \|_{L^\infty(\R)} \leq c \| f_\omega - h_\omega \|_{L^2(\R)}
$$
for some constant $c>0$.

Moreover, for a fixed $\omega \in \R$, Corollary \ref{cor:tensor_stability} shows that the tensor product $f_\omega$ of functions $f \in V_\beta^1(\varphi)$ is stably determined by spectrogram measurements and the stability constant $\tfrac{1}{A(\sigma,\beta,\omega)}$ in inequality \eqref{ineq:stability_tensor_product} is independent of $f$ and $h$. Consequently, the reconstruction of $f_\omega$ from $|\mathcal{G}f(\tfrac{\beta}{2}\Z \times \R)|$ is globally stable within the signal class $V_\beta^1(\varphi)$. Since
$$
A(\sigma,\beta,\omega) = \frac{\sigma}{\sqrt 2} \varphi(\omega) \vartheta_3(\tfrac{\pi}{2},\varphi(\tfrac{\beta}{2}))
$$
it follows that the stability constant for reconstructing the modulus $f_0=|f|^2$ is the best among the constants $\{ A(\sigma,\beta,\omega) : \omega \in \R \}$ and the stability deteriorates rapidly in $\omega$.
Classical examples for instabilities are of the form
\begin{equation}\label{example:instability}
    \begin{split}
        f_{n,+} &= T_{\beta n} \varphi + T_{-\beta n} \varphi, \\
        f_{n,-} &= T_{\beta n} \varphi - T_{-\beta n} \varphi,
    \end{split}
\end{equation}
and they show that the stability constant grows like $e^{n^2}$ \cite{ALAIFARI2021401}. Thus, one cannot expect to stably reconstruct arbitrary non-positive functions without imposing further assumptions.
We surpass instabilities arising from signals of the form \eqref{example:instability} in the following way:
A closer look on the map $f \coloneqq f_{n,+}$ shows that on the interval $I = [-\beta n , \beta n]$ it holds that
$$
|f_{n,+}(\beta n)| = |f_{n,+}(-\beta n)| \approx 1
$$
whereas $|f| \approx 0$ on a large part of the interior of $I$. Therefore, $|f|$ (and also $|\mathcal{G}f|$) is concentrated on two almost disjoint components which is a classical indicator for instabilities \cite{daub,GrohsRathmair,GrohsRathmair2}. Such situations can be excluded by the requirement that $f$ satisfies condition \textbf{(P)}: if $p_1 < \cdots < p_J \in \R, I \subseteq [p_1,p_J]$ and $\gamma > 0$ with $|f(p_j)| \geq \gamma$ for all $j$ then the distance of the areas where $|f|$ is concentrated is determined by the value $r = \max_{1 \leq j \leq J-1} p_{j+1}-p_j$. In this setting we expect that the quotient distance $\min_{\tau \in \T} \| f - \tau g \|$ can be controlled by the norm of the difference of the corresponding spectrograms up to factor which depends on $e^{r^2}$ and $\frac{1}{\gamma}$. This is indeed the case as we shall elaborate in the following.
First, we define a constant $C(\sigma,\beta)$ depending only on the variance of the generator as well as the step-size of the underlying Gaussian shift-invariant space by
\begin{equation}\label{def:stability_constant_C}
    C(\sigma,\beta) \coloneqq \sup_{p \in \R} \sum_{n \in \Z} |T_{\frac{\beta}{2}n}\ift \Lambda(p)|
\end{equation}
where $\Lambda$ is the map which characterizes the dual generator, see equation \eqref{eq:Lambda}. We start with local stability estimates around points $p_j \in \R, 1 \leq j \leq J, J \in \N$ where $f(p_j) \neq 0$ for all $j$.

\begin{lemma}\label{lma:local_stability}
Let $f,g \in V_\beta^\infty(\varphi), r>0, J \in \N$ and $p_j \in \R, 1 \leq j \leq J,$ such that $f(p_j),g(p_j) \neq 0$ for every $j$. Define
$$
\tau_j \coloneqq \frac{\overline{g(p_j)}|f(p_j)|}{|g(p_j)|\overline{f(p_j)}}, \ I_j \coloneqq [p_j-r,p_j+r],
$$
$$
c_j \coloneqq \frac{1}{|g(p_j)|} \left ( e^{\frac{r^2}{4\sigma^2}} + \frac{\| f \|_{L^\infty(I_j)}}{|f(p_j)|+|g(p_j)|} \right ).
$$
Then $\| |f|^2-|g|^2 \|_{L^\infty(\R)} \leq \sqrt{2} \| |\mathcal{G}f|^2-|\mathcal{G}g|^2 \|_{\frac{\beta}{2},\infty} C(\sigma,\beta)$ and for every $j \in \{ 1, \dots, J \}$ it holds that $$\min_{\tau \in \T} \| f-\tau g \|_{L^\infty(I_j)} \leq \| f-\tau_j g \|_{L^\infty(I_j)} \leq c_j \sqrt{2} \| |\mathcal{G}f|^2-|\mathcal{G}g|^2 \|_{\frac{\beta}{2},\infty} C(\sigma,\beta),$$
with $C(\sigma,\beta)$ being defined as in equation \eqref{def:stability_constant_C}.
\end{lemma}

\begin{proof}
In view of Corollary \ref{cor:tp_expansion}, both maps $|f|^2$ and $|g|^2$ can be expanded in terms of the dual generator of $\varphi_0 = \varphi^2$. Hence, we have for every $t \in \R$ the estimate
\begin{equation*}
    \begin{split}
        ||f(t)|^2-|g(t)|^2| & = \left | \sum_n \langle f_0,T_{\frac{\beta}{2}n}\varphi_0 \rangle T_{\frac{\beta}{2}n} \widetilde{\varphi_0}(t)  - \sum_n \langle g_0,T_{\frac{\beta}{2}n}\varphi_0 \rangle T_{\frac{\beta}{2}n} \widetilde{\varphi_0}(t)    \right | \\
        & \leq \sum_n |\langle f_0,T_{\frac{\beta}{2}n}\varphi_0 \rangle-\langle g_0,T_{\frac{\beta}{2}n}\varphi_0 \rangle||T_{\frac{\beta}{2}n} \widetilde{\varphi_0}(t)| \\
        & \leq
 \sqrt{2} \| |\mathcal{G}f|^2-|\mathcal{G}g|^2 \|_{\frac{\beta}{2},\infty}  \sum_n |T_{\frac{\beta}{2}n}\ift \Lambda(p)| \\
 & \leq \sqrt{2} \| |\mathcal{G}f|^2-|\mathcal{G}g|^2 \|_{\frac{\beta}{2},\infty} C(\sigma,\beta) =: B.
    \end{split}
\end{equation*}
Let $\omega \in \R$ with $|\omega|\leq r$. Since $f(p_j),g(p_j) \neq 0$ the explicit reconstruction formula from tensor products shows that 
$$
f(p_j+\omega) = \frac{1}{\overline{f(p_j)}} \overline{f_\omega(p_j+\omega)}, \ g(p_j+\omega) = \frac{1}{\overline{g(p_j)}} \overline{g_\omega(p_j+\omega)}.
$$
The choice of $\tau_j$ implies that
\begin{equation*}
    \begin{split}
        & |f(p_j+\omega)-\tau_j g(p_j+\omega)| = \left | \frac{1}{|f(p_j)|} f_\omega(p_j+\omega) - \frac{1}{|g(p_j)|} g_\omega(p_j+\omega) \right | \\
        & \leq \frac{1}{|g(p_j)|} |f_\omega(p_j+\omega) - g_\omega(p_j+\omega)| +  \frac{||g(p_j)|-|f(p_j)||}{|f(p_j)g(p_j)|} |f_\omega(p_j+\omega)|  \\
        & = \frac{1}{|g(p_j)|}  ( |f_\omega(p_j+\omega) - g_\omega(p_j+\omega)| + ||f(p_j)|-|g(p_j)|||f(p_j+\omega)| ) \\
        & \leq \frac{1}{|g(p_j)|}  \left ( |f_\omega(p_j+\omega) - g_\omega(p_j+\omega)| +  \frac{ B}{|f(p_j)|+|g(p_j)|} \| f \|_{L^\infty(I_j)}\right ).
    \end{split}
\end{equation*}
Setting $A \coloneqq |f_\omega(p_j+\omega) - g_\omega(p_j+\omega)|$ and employing the reconstruction formula from Theorem \ref{thm:explicit_reconstruction} and the explicit form of the dual generator gives
\begin{equation*}
    \begin{split}
        A & = \left | \sum_n \left (  \int_\R (|\mathcal{G}f(\tfrac{\beta}{2}n,t)|^2-|\mathcal{G}g(\tfrac{\beta}{2}n,t)|^2)e^{-2\pi i \omega t} \, dt \right)T_{\frac{\beta}{2}n}\widetilde{\varphi_\omega}(p+\omega) \right | \\
        & \leq \sum_n \| |\mathcal{G}f(\tfrac{\beta}{2}n,\cdot)|^2 - |\mathcal{G}g(\tfrac{\beta}{2}n,\cdot)|^2 \|_{L^1(\R)} |T_{\frac{\beta}{2}n}\widetilde{\varphi_\omega}(p+\omega)| \\
        & \leq \sqrt{2}e^{\frac{\omega^2}{4\sigma^2}} \| |\mathcal{G}f|^2-|\mathcal{G}g|^2 \|_{\frac{\beta}{2},\infty} \sum_n |T_{\frac{\beta}{2}n+\frac{\omega}{2}}\ift \Lambda(p+\omega)| \\
        & \leq e^{\frac{r^2}{4\sigma^2}}B.
    \end{split}
\end{equation*}
The assertion follows from that.
\end{proof}

For $j \in \N$, let $\tau_j,I_j$ and $c_j$ be defined analogue to Lemma \ref{lma:local_stability}. The following elementary Lemma will be useful in the proof of the main stability result of the present section.

\begin{lemma}\label{lma:phase_bound}
Let $f,g \in V_\beta^\infty(\varphi), r>0$ and $p_j,p_m \in \R$ such that $$f(p_j),g(p_j),f(p_m),g(p_m) \neq 0.$$ Suppose that $c_j, \tau_j, \tau_m$ are defined as in Lemma \ref{lma:local_stability}. If we define
\begin{equation}\label{BB}
    B \coloneqq \sqrt{2} \| |\mathcal{G}f|^2-|\mathcal{G}g|^2 \|_{\frac{\beta}{2},\infty} C(\sigma,\beta)
\end{equation}
then we have the implication
$$
|p_j - p_m| \leq r \implies |\tau_j-\tau_m| \leq \frac{c_jB}{|g(p_m)|} + \frac{B}{|f(p_m)|(|f(p_m)+|g(p_m)|)}.
$$
\end{lemma}
\begin{proof}
Since $|p_j-p_m| \leq r$ we have $p_m \in I_j$. Lemma \ref{lma:local_stability} implies that
$$
|f(p_m)-\tau_j g(p_m)| \leq c_j B.
$$
Using the definition of $\tau_j$ we can re-write the left-hand side of the previous inequality as
$$
|f(p_m)-\tau_j g(p_m)| = \left |  \frac{f(p_m)\overline{f(p_j)}|g(p_j)|}{|f(p_j)g(p_j)|} - \frac{g(p_m)\overline{g(p_j)}|f(p_j)|}{|f(p_j)g(p_j)|}    \right | =: |\mu-\nu|.
$$
Since
$$
\tau_j = \frac{\overline{g(p_j)}|f(p_j)|}{|g(p_j)|\overline{f(p_j)}}, \ \ \tau_m = \frac{\overline{g(p_m)}|f(p_m)|}{|g(p_m)|\overline{f(p_m)}}
$$
we obtain with Lemma \ref{lma:local_stability}
\begin{equation*}
    \begin{split}
        |\tau_j-\tau_m| & = \left | \frac{1}{|f(p_m)|} \mu - \frac{1}{|g(p_m)|} \nu  \right | 
         \leq \frac{|\mu-\nu|}{|g(p_m)|} + \frac{||f(p_m)|^2-|g(p_m)|^2|}{|g(p_m)|(|f(p_m)|+|g(p_m)|)} \\
        & \leq \frac{c_j B}{|g(p_m)|} + \frac{B}{|g(p_m)|(|f(p_m)|+|g(p_m)|)}.
    \end{split}
\end{equation*}
\end{proof}

With the help of Lemma \ref{lma:phase_bound}, the local stability result obtained in Lemma \ref{lma:local_stability} transfers to stability estimates on larger intervals by imposing condition \textbf{(P)}.

\begin{theorem}\label{thm:main_stability}
Let $\gamma > 0$ and $p_1 < \cdots < p_J \in \R$ such that $f \in V_\beta^\infty(\varphi)$ satisfies condition \textbf{(P)}. If $I = [p_1-r,p_J+r]$ and $r = \max_{1 \leq j \leq J-1} (p_{j+1}-p_j)$ then for every $g \in V_\beta^\infty(\varphi)$ we have
\begin{equation}\label{stabilityBoundthm}
\begin{split}
& \min_{\tau \in \T} \| f-\tau g \|_{L^\infty(I)} \\ & \lesssim_{\beta, \sigma} (J-1)e^{\frac{r^2}{4\sigma^2}} \frac{\max \{ \| f \|_{L^\infty(I)}^2, \| f \|_{L^\infty(I)}+\| g \|_{L^\infty(I)} \} }{\min \{ \gamma,\gamma^3 \}} \| |\mathcal{G}f|^2-|\mathcal{G}g|^2 \|_{\frac{\beta}{2},\infty}
\end{split}
\end{equation}
where the implicit constant is upper bounded by $\frac{32\sqrt{2}}{3} C(\sigma,\beta)$.
\end{theorem}
\begin{proof}
\textbf{Case 1.} Suppose first that $|\mathcal{G}g|^2$ is in a ball of radius $(2\sqrt{2}C(\sigma,\beta))^{-1}\gamma^2$ around $|\mathcal{G}f|^2$ with respect to the mixed norm $\| \cdot \|_{\frac{\beta}{2},\infty}$, that is,
$$
\| |\mathcal{G}f|^2-|\mathcal{G}g|^2 \|_{\frac{\beta}{2},\infty} \leq \frac{\gamma^2}{2\sqrt{2}C(\sigma,\beta)}.
$$
By Lemma \ref{lma:local_stability} we have $\| |f|^2-|g|^2 \|_{L^\infty(\R)} \leq  \sqrt{2} C(\sigma,\beta) \| |\mathcal{G}f|^2-|\mathcal{G}g|^2 \|_{\frac{\beta}{2},\infty}$ and therefore for every $j \in \{ 1,\dots, J \}$ it holds that
$$
||f(p_j)|-|g(p_j)| \leq \frac{||f(p_j)|^2-|g(p_j)|^2|}{|f(p_j)|} \leq \frac{\gamma}{2}.
$$
In particular, $|g(p_j)| \geq \frac{\gamma}{2}$.
Define $I_j \coloneqq [p_j-r,p_j+r], j \in \{ 1,\dots, J \}$. Since $|p_{j+1}-p_j| \leq r$ for all $j$ we have
$
I = \bigcup_{j=1}^J I_j.
$
Now let $\tau_1, \dots, \tau_J \in \T$ be chosen as in Lemma \ref{lma:local_stability}. Then
\begin{equation}\label{eq:l_inf_split}
    \| f-\tau_1g \|_{L^\infty(I)} = \max \{ \| f-\tau_1 g \|_{L^\infty(I_j)} : j=1,\dots, J \}.
\end{equation}
and for a fixed $j \in \{ 1,\dots, J \}$ we have
\begin{equation*}
    \begin{split}
        \| f-\tau_1 g \|_{L^\infty(I_j)} & \leq \| f-\tau_j g \|_{L^\infty(I_j)} + |\tau_1 - \tau_j| \| f \|_{L^\infty(I_j)} \\
        & \leq \| f-\tau_j g \|_{L^\infty(I_j)} + \| f \|_{L^\infty(I)} \sum_{i=1}^{j-1} |\tau_i - \tau_{i+1}| \\
        & \leq \frac{2 B}{\gamma} \left ( e^{\frac{r^2}{4\sigma^2}} + \frac{2\| f \|_{L^\infty(I)}}{3\gamma} \right ) + \| f \|_{L^\infty(I)} \sum_{i=1}^{j-1} |\tau_i - \tau_{i+1}|.
    \end{split}
\end{equation*}
By the assumption on the distance between the $p_j$'s and the definition of the $I_j$'s we have $p_j \in I_{j-1}$ for every $j \in \{ 2,\dots, J \}$. Therefore, Lemma \ref{lma:phase_bound} shows that
\begin{equation*}
    \begin{split}
        \sum_{i=1}^{j-1} |\tau_i - \tau_{i+1}| & \leq \sum_{i=1}^{j-1} \left ( \frac{c_iB}{|g(p_{i+1})|} + \frac{B}{|f(p_{i+1})|(|f(p_{i+1})|+|g(p_{i+1})|)} \right) \\
        & \leq \sum_{i=1}^{j-1} \frac{2c_iB}{\gamma} + \frac{2B}{3\gamma^2} \\
        & \leq \frac{2(j-1)B}{3\gamma^2} + \frac{4(j-1)B}{\gamma^2} \left ( e^{\frac{r^2}{4\sigma^2}} + \frac{2\| f \|_{L^\infty(I)}}{3\gamma} \right )
    \end{split}
\end{equation*}
where $B$ is defined as in \eqref{BB}.
Using standard estimates it follows that,
$$
\| f-\tau_1 g \|_{L^\infty(I_j)} \leq \frac{32}{3} (j-1) e^{\frac{r^2}{4\sigma^2}} \frac{\max \{ \| f \|_{L^\infty(I)}, \| f \|_{L^\infty(I)}^2 \} }{\min \{ \gamma,\gamma^2,\gamma^3 \}} C(\sigma, \beta) B
$$
which implies in combination with equation \eqref{eq:l_inf_split} that
$$
\| f-\tau_1 g \|_{L^\infty(I)} \leq \frac{32}{3} (J-1) e^{\frac{r^2}{4\sigma^2}} \frac{\max \{ \| f \|_{L^\infty(I)}, \| f \|_{L^\infty(I)}^2 \} }{\min \{ \gamma,\gamma^2,\gamma^3 \}} B.
$$
\textbf{Case 2.} 
If
$$
\| |\mathcal{G}f|^2-|\mathcal{G}g|^2 \|_{\frac{\beta}{2},\infty} > \frac{\gamma^2}{2\sqrt{2}C(\sigma,\beta)}
$$
then clearly for every $\tau \in \T$ we have the estimate
$$
\| f-\tau g \|_{L^\infty(I)} \leq \frac{2(\| f \|_{L^\infty(I)} + \| g \|_{L^\infty(I)})}{\gamma^2} B.
$$
The statement follows from case 1 and case 2.
\end{proof}

Additional assumptions on $\sigma$ and $\beta$ yield explicit upper bounds on the constant $C=C(\sigma,\beta)$. To achieve such bounds we first observe that according to Section \ref{subsection:gaussian_generators} the map $\ift \Lambda$ satisfies an exponential decay. Thus, there are constants $K=K(\sigma,\beta), \nu=\nu(\sigma,\beta)>0$ such that
\begin{equation}\label{decay_dual_generator}
    |\ift \Lambda(t)| \leq Ke^{-\nu|t|}.
\end{equation}

In an exemplary way we can derive a result of the following form which will result in explicit upper bounds on $C(\sigma,\beta)$.

\begin{lemma}\label{lma:expl_decay}
If $\frac{\beta}{4} \leq \sigma \leq \frac{\beta}{2} \leq 1$ then $\nu \geq \frac{1}{4}$ and $K \leq \frac{205}{\sigma}$.
\end{lemma}
\begin{proof}
See the Appendix \ref{appendix:C}.
\end{proof}

A consequence of Lemma \ref{lma:expl_decay} is an explicit upper bound on the stability constant $C(\sigma,\beta)$: if $\frac{\beta}{4} \leq \sigma \leq \frac{\beta}{2} \leq 1$ then
\begin{equation*}
     C(\sigma,\beta) = \sup_{p \in \R} \sum_{n \in \Z} |\ift \Lambda(p-\tfrac{\beta}{2})| \leq  \frac{205}{\sigma} \left ( 2 + \int_\R e^{-\frac{\beta}{8}|t|} \right ) = \frac{205}{\sigma} \left ( 2+\frac{16}{\beta} \right ).
\end{equation*}

\begin{remark}[On the dependence on $\| f \|_{L^\infty(I)}$ and $\| g \|_{L^\infty(I)}$ in the stability estimte \eqref{stabilityBoundthm}]

Shift-invariant spaces $V_\alpha^p(\varphi^u), \alpha>0, p\in[1,\infty],$ generated by a Gaussian with standard deviation $u>0$ satisfy the norm equivalence
$$
\| h \|_{L^p(\R)} \asymp \| a \|_{\ell^p(\Z)}, \ h \in V_\alpha^p(\varphi^\nu),
$$
where $h$ has defining sequence $a=(a_n)_n \in \ell^p(\Z)$ \cite[Theorem 3.1]{GroechenigSurvey}.
Moreover, recent sampling results for shift-invariant spaces with totally positive generator of Gaussian type show that
$$
\| h \|_{L^p(\R)} \asymp \| h (\Omega) \|_{\ell^p(\Z)}
$$
whenever $\Omega \subset \R$ is separated with lower Beurling density $D^-(\Omega) > \alpha^{-1}$ \cite[Theorem 4.4]{GroechenigRomeroStoeckler}. 
In equation \eqref{frequency_zero} we observed that if $f \in V_\beta^\infty(\varphi)$ has defining sequence $c \in \ell^\infty(\Z)$ then $\mathcal{G}f(\cdot,0) \in V_\beta^\infty(\varphi^{\sqrt{2}\sigma})$ has defining sequence $(d_n)_n$ with $d_n = \sigma \sqrt{\pi} c_n$. Setting $\Omega = \frac{\beta}{2}\Z$ gives $D^-(\Omega) = 2 \beta^{-1}>\beta^{-1}$ and the norm equivalences above imply
$$
\| f \|_{L^\infty(\R)} \asymp \| |\mathcal{G}f| \|_{L^\infty(\R \times \{ 0 \})} \asymp \| |\mathcal{G}f| \|_{\ell^\infty(\frac{\beta}{2}\Z \times \{ 0 \})}.
$$
Consequently, one has
$$
\| f \|_{L^\infty(I)} \lesssim_{\sigma,\beta} \| |\mathcal{G}f| \|_{\ell^\infty(\frac{\beta}{2}\Z \times \{ 0 \})}
$$
with a universal constant depending only on $\sigma$ and $\beta$. With this choice of a sampling set $\Omega$, the right-hand side of the stability estimate \eqref{stabilityBoundthm} can be made dependent only on $\gamma, \sigma, \beta, |\mathcal{G}f(\frac{\beta}{2}\Z \times \R)|$ and $|\mathcal{G}g(\frac{\beta}{2}\Z \times \R)|$.
\end{remark}

\section{Algorithmic approximation in $V_\beta^\infty(\varphi)$}\label{sec:algo}

In the final section of the present article we shall combine our previous results to derive an algorithmic approximation of functions $f \in V_\beta^\infty(\varphi)$ from finitely many spectrogram samples $|\mathcal{G}f(X)|, X \subset \R^2, |X| < \infty$. Assume that the samples $X$ are given on a finite grid of the form
\begin{equation}\label{def:grid_X}
X = \tfrac{\beta}{2} \{ -N, \dots, N \} \times h \{ -H, \dots, H \}
\end{equation}
where $N,H \in \N$ and $h,\beta >0$. We call $\frac{1}{h}$ the sampling density of $X$ in frequency direction. If $\eta \in \R^{(2N+1) \times (2H+1)}$ denotes a noise matrix then the given sampling set $\mathfrak{S}$ takes the form
\begin{equation*}
    \begin{split}
        & \mathfrak{S}=(|\mathcal{G}f(X)|^2+\eta_{n,k})_{n,k} \in \R^{(2N+1) \times (2H+1)}, \\
        & \mathfrak{S}(n,k) = |\mathcal{G}f(\tfrac{\beta}{2}n,hk)|^2 + \eta_{n,k}.
    \end{split}
\end{equation*}
To quantify the contribution of the noise $\eta$ we use the $\ell^\infty$-operator norm of $\eta$,
$$
\| \eta \|_\infty \coloneqq  \sup_{x \neq 0} \frac{\| \eta x \|_\infty}{\|x\|_\infty} = \max_{1 \leq i \leq 2N+1} \sum_{j=1}^{2H+1} |\eta_{ij}|,
$$
which is the maximum absolute row sum of $\eta$. Our only assumption on the noise $\eta$ is that we have control over the norm $\| \eta \|_\infty \leq \delta$ for some noise-level $\delta > 0$. Otherwise, the error is arbitrary and can be adversarial.

\begin{figure}
\centering
   \includegraphics[width=12.5cm]{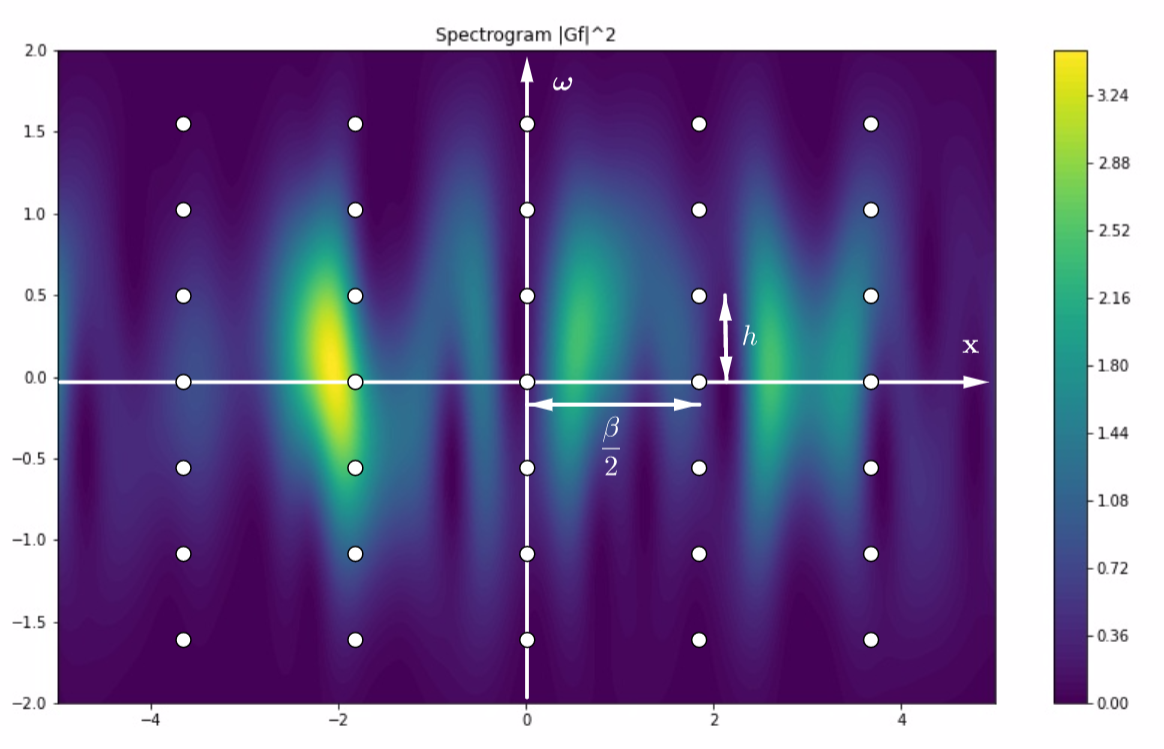}
\caption{The contour plot depicts the spectrogram of a function $f \in V_\beta^\infty(\varphi)$ and the white dots visualize the location of the samples $\mathfrak{S}$. The values $\frac{2}{\beta}$ and $\frac{1}{h}$ are the densities in time ($x$) resp. frequency ($\omega$) direction.}
\end{figure}

\subsection{The numerical approximation routine $\mathcal{R}$}\label{sec:reconstructive_function}

We start by presenting our reconstruction approach in an abstract setting, leaving the ansatz open for extensions to other signal classes (see Section \ref{sec:beyond_V}). Assume for the time-being that the function $f$ has only the property of being measurable and bounded. Let $p_1 < \cdots < p_J$ and $\gamma>0$ such that $f$ satisfies condition \textbf{(P)}. For simplicity we assume that $-p_1=p_J \eqqcolon s$ and set $I=[-s,s]$. The aim is to recover $f$ on the interval $I$. Generalizations to other intervals are straightforward and can be achieved by a simple translation. 
Motivated by Proposition \ref{prop:tensorproduct_properties} and Theorem \ref{thm:explicit_reconstruction} we define constants $c_j$ by
$$
c_j = h\sum_n \sum_k \mathfrak{S}(n,k) T_{\frac{\beta}{2}n} \widetilde{\varphi_0}(p_j),
$$
local reconstructive functions $L_j$ by
$$
L_j(\omega) = \frac{h}{\sqrt{c_j}} \sum_n \sum_k \mathfrak{S}(n,k)e^{2\pi i \omega h k} T_{\frac{\beta}{2}n} \widetilde{\varphi_\omega}(p_j+\omega)
$$
and phases $\nu_0, \dots, \nu_{J-1} \in \T$ by
$$
\nu_0 = 1, \ \ \nu_j = \frac{L_j(p_{j+1}-p_j)}{|L_j(p_{j+1}-p_j)|} \ (j=1,\dots, J-1).
$$
Note that $c_j, L_j$ and $\nu_j$ depend only on the sampling set $\mathfrak{S}$.
Assuming that $L_j$ and $\nu_j$ are well-defined, i.e. $c_j >0$ and $L_j(p_{j+1}-p_j) \neq 0$, gives rise to the definition of a numerical approximation routine.

\begin{definition}\label{def:num_approx_routine}
Suppose that $f : \R \to \C$ is a measurable and bounded map which satisfies condition \textbf{(P)} with $\gamma>0$ and $p_1 < \cdots < p_J \in \R, -p_1=p_J=s$. Let the grid $X$ be defined as in \eqref{def:grid_X} and suppose that $L_j$ and $\nu_j$ are well-defined. The numerical approximation routine $\mathcal{R} : [-s,s] \to \C$ of $f$ is defined as follows: 
if $t \in [-s,s]$ such that $t \in (p_j,p_{j+1}]$ with $t=p_j + \omega$ then
$$
\mathcal{R}(t) = \nu_1 \cdots \nu_{j-1} L_j(\omega).
$$
For $t=p_1=-s$ we define $\mathcal{R}(t)=L_1(0)$.
\end{definition}

In order to derive approximation results by means of the numerical approximation routine $\mathcal{R}$, we define for $j = 1, \dots, J-1$ the error term
\begin{equation}\label{eq:error_term}
    E_j(\omega) \coloneqq |f_\omega(p_j+\omega) - \sqrt{c_j}\overline{L_j(\omega)}|.
\end{equation}
This error term allows us to state conditions under which the functions $L_j$ and the phases $\nu_j$ are well-defined. More importantly, it yields an upper bound on the quotient distance of $f$ and  $\mathcal{R}$.

\begin{theorem}\label{thm:abstract_bound}
Let $f : \R \to \C$ be a measurable and bounded function which satisfies condition \textbf{(P)} with $\gamma>0$ and $p_1 < \cdots < p_J \in \R, -p_1=p_J=s$. If $I=[-s,s]$, $0 < \varepsilon \leq \min \left \{ \frac{\gamma^2}{2\sqrt{8}}, \frac{\gamma^3}{4 \| f \|_{L^\infty(I)}} \right \}$ and
$$
\max_{1\leq j \leq J-1} \| E_j \|_{L^\infty[0,p_{j+1}-p_j]} \leq \varepsilon
$$
then $L_j$ and $\nu_j$ are well-defined for every $j\in\{ 1, \dots, J \}$. Further, it holds that
$$
\min_{\tau \in \T} \| f - \tau \mathcal{R} \|_{L^\infty(I)}  \leq 32(J-1)\frac{\max \{ 1, \| f \|^2_{L^\infty(I)} \}}{\min \{ \gamma, \gamma^5 \}} (\varepsilon + \varepsilon^2).
$$
\end{theorem}

\begin{proof}
Let $I \coloneqq [-s,s]$ and $I_j \coloneqq (p_j,p_{j+1}]$ for $j=1, \dots, J-1$.

\textbf{Step 1: Upper bounding $\min_{\tau \in \T} \| f - \tau \mathcal{R} \|_{L^\infty(I)}$.} Assume for the time-being that the phases $\nu_j$ and the constants $c_j$ are well-defined.
Let $t \in I$ such that $t \in (p_j,p_{j+1}]$ with $t = p_j + \omega$ and $2 \leq j \leq J-1$. Then
\begin{equation}\label{eq:14}
\begin{split}
& \left | f(t) - \frac{f(p_1)}{|f(p_1)|} \mathcal{R}(t) \right | \\
& \leq \left | f(p_j + \omega) - \frac{f(p_j)}{|f(p_j)|} L_j(\omega) \right | + \left | \frac{f(p_j)}{|f(p_j)|} L_j(\omega) - \frac{f(p_1)}{|f(p_1)|} \nu_1 \cdots \nu_{j-1} L_j(\omega)   \right | \\
& =  \left | f(p_j + \omega) - \frac{f(p_j)}{|f(p_j)|} L_j(\omega) \right | + \underbrace{\left | \frac{f(p_j)}{|f(p_j)|}  - \frac{f(p_1)}{|f(p_1)|} \nu_1 \cdots \nu_{j-1}   \right |}_{\coloneqq A} |L_j(\omega)|.
\end{split}
\end{equation}
The term $A$ can be bounded by
\begin{equation}\label{eq:15}
\begin{split}
A & = \left | \overline{\nu_{j-1}} \frac{f(p_j)}{|f(p_j)|}  - \frac{f(p_1)}{|f(p_1)|} \nu_1 \cdots \nu_{j-2}   \right | \\
& \leq \left | \overline{\nu_{j-1}} \frac{f(p_j)}{|f(p_j)|} - \frac{f(p_{j-1})}{|f(p_{j-1})|} \right | + \left | \frac{f(p_{j-1})}{|f(p_{j-1})|} -  \frac{f(p_1)}{|f(p_1)|} \nu_1 \cdots \nu_{j-2}   \right | \\
& \leq \sum_{i=2}^j \left |  \frac{f(p_i)}{|f(p_i)|} - \frac{f(p_{i-1})}{|f(p_{i-1})|} \nu_{i-1} \right |
\end{split}
\end{equation}
where the third inequality follows by induction. Each summand in the last term of the sum above satisfies
\begin{equation}\label{eq:16}
\begin{split}
& \left |  \frac{f(p_i)}{|f(p_i)|} - \frac{f(p_{i-1})}{|f(p_{i-1})|} \nu_{i-1} \right | \\
& = \left |  \frac{1}{|f(p_i)|}f(p_i) - \frac{1}{|L_{i-1}(p_i - p_{i-1})|} \frac{f(p_{i-1})}{|f(p_{i-1})|}  L_{i-1}(p_i - p_{i-1}) \right | \\
& \leq |f(p_i)| \left | \frac{1}{|f(p_i)|} - \frac{1}{|L_{i-1}(p_i - p_{i-1})|} \right | \\ &  + \frac{1}{|L_{i-1}(p_i - p_{i-1})|} \left | f(p_i) - \frac{f(p_{i-1})}{|f(p_{i-1})|}  L_{i-1}(p_i - p_{i-1}) \right | \\
& = |f(p_i)| \left | \frac{|f(p_i)|-|\frac{f(p_{i-1})}{|f(p_{i-1})|}L_{i-1}(p_i - p_{i-1})|}{|f(p_i)||L_{i-1}(p_i - p_{i-1})|}  \right | \\ &  + \frac{1}{|L_{i-1}(p_i - p_{i-1})|} \left | f(p_i) - \frac{f(p_{i-1})}{|f(p_{i-1})|}  L_{i-1}(p_i - p_{i-1}) \right | \\
& \leq  \frac{2}{|L_{i-1}(p_i - p_{i-1})|} \left | f(p_i) - \frac{f(p_{i-1})}{|f(p_{i-1})|}  L_{i-1}(p_i - p_{i-1}) \right |.
\end{split}
\end{equation}
Combining the inequalties \eqref{eq:14}, \eqref{eq:15} and \eqref{eq:16} gives

\begin{equation*}
    \begin{split}
        & \left | f(t) - \frac{f(p_1)}{|f(p_1)|} \mathcal{R}(t) \right | \leq \left | f(p_j + \omega) - \frac{f(p_j)}{|f(p_j)|} L_j(\omega) \right | \\
        & + 2|L_j(\omega)| \sum_{i=2}^j \frac{1}{|L_{i-1}(p_i - p_{i-1})|} \left | f(p_i) - \frac{f(p_{i-1})}{|f(p_{i-1})|}  L_{i-1}(p_i - p_{i-1}) \right |.
    \end{split}
\end{equation*}

Observe that $p_i \in I_{i-1}$ and $f(p_i)=f(p_{i-1}+(p_i - p_{i-1}))$. Therefore, if we set
$$Q \coloneqq \max_{1 \leq j \leq J-1} \| f(p_j + \cdot) - \tfrac{f(p_j)}{|f(p_j)|} L_j \|_{L^\infty(I_j)}$$
then
$$
\left | f(t) - \frac{f(p_1)}{|f(p_1)|} \mathcal{R}(t) \right | \leq Q + 2Q|L_j(\omega)|\sum_{i=2}^j \frac{1}{|L_{i-1}(p_i - p_{i-1})|}.
$$
Observing that
$$
|L_j(\omega)|  \leq \left | f(p_j + \omega) - \frac{f(p_j)}{|f(p_j)|} L_j(\omega) \right | + |f(p_j+\omega)| \leq Q + \|f \|_{L^\infty(I)},
$$
gives
\begin{equation}\label{eq:rx_bound}
    \min_{\tau \in \T} \| f - \tau \mathcal{R} \|_{L^\infty(I)} \leq  Q + 2Q\left (\|f \|_{L^\infty(I)}+Q \right )\sum_{i=2}^j \frac{1}{|L_{i-1}(p_i - p_{i-1})|}.
\end{equation}

\textbf{Step 2. Well-definedness of $c_j$.}
If the map $S_j$ is defined by
$$
S_j(\omega) = \sqrt{c_j} \overline{L_j(\omega)} = h\sum_n \sum_k \mathfrak{S}(n,k)e^{-2\pi i \omega h k} T_{\frac{\beta}{2}n} \widetilde{\varphi_\omega}(p_j+\omega),
$$
then
$
E_j(\omega) = |f_\omega(p_j+\omega) - S_j(\omega)|.
$
Observe that $c_j = S_j(0), f_0(p_j) = |f(p_j)|^2$ and
$$
c_j \geq |f(p_j)|^2-||f(p_j)|^2-c_j| \geq \gamma^2 -  |f_0(p_j) - S_j(0)| = \gamma^2 - E_j(0).
$$
By assumption we have $E_j(0) \leq \max_{1 \leq j \leq J-1} \| E_j \|_{L^\infty[0,p_{j+1}-p_j]} \leq \varepsilon \leq \frac{\gamma^2}{2}$ which gives
\begin{equation}\label{c_lower_bound}
    c_j \geq \frac{\gamma^2}{2}.
\end{equation}
In particular, the constants $c_j$ are well-defined.

\textbf{Step 3: Upper bounding $E_j(\omega)$.}
We have shown that the constants $c_j$ are well-defined. Recalling that if $f(p_j) \neq 0$ then the identity $\frac{1}{\overline{f(p_j)}} \overline{f_\omega(p_j + \omega)}$ holds for every $\omega \in \R$. We can continue with the following estimate:
\begin{equation}\label{upper_bound_E}
\begin{split}
& | f(p_j + \omega) - \tfrac{f(p_j)}{|f(p_j)|} L_j(\omega) | 
 = \left | \frac{1}{\overline{f(p_j)}} \overline{f_\omega(p_j + \omega)} - \frac{f(p_j)}{|f(p_j)|} \frac{1}{\sqrt{c_j}} \overline{S_j(\omega)} \right | \\
& = \left | \frac{1}{f(p_j)} f_\omega(p_j + \omega) - \frac{|f(p_j)|}{f(p_j)} \frac{1}{\sqrt{c_j}} S_j(\omega) \right | \\
& \leq \frac{1}{\sqrt{c_j}} |f_\omega(p_j+\omega) - S_j(\omega)| + \left | \frac{f(p_j)}{|f(p_j)|} \left ( \frac{1}{f(p_j)} - \frac{|f(p_j)|}{f(p_j)}\frac{1}{\sqrt{c_j}} \right) \right | |f_\omega(p_j + \omega)| \\
& = \frac{1}{\sqrt{c_j}} |f_\omega(p_j+\omega) - S_j(\omega)| + \left | \frac{1}{|f(p_j)|} - \frac{1}{\sqrt{c_j}} \right | |f_\omega(p_j + \omega)| \\
& = \frac{1}{\sqrt{c_j}} |f_\omega(p_j+\omega) - S_j(\omega)| + \frac{|c_j - |f(p_j)|^2|}{\sqrt{c_j}(\sqrt{c_j}+|f(p_j)|)} |f(p_j + \omega)| \\
& \leq  \frac{1}{\sqrt{c_j}} |f_\omega(p_j+\omega) - S_j(\omega)| + \frac{\| f \|_{L^\infty(I)}}{\sqrt{c_j}(\sqrt{c_j}+|f(p_j)|)} |f_0(p_j) - S_j(0)| \\
& = \frac{1}{\sqrt{c_j}} E_j(\omega) + \frac{\| f \|_{L^\infty(I)}}{\sqrt{c_j}(\sqrt{c_j}+|f(p_j)|)} E_j(0).
\end{split}
\end{equation}
Using the bound obtained in equation \eqref{c_lower_bound} together with the fact that $|f(p_j)| \geq \gamma$ yields
\begin{equation}\label{eq:E}
    | f(p_j + \omega) - \tfrac{f(p_j)}{|f(p_j)|} L_j(\omega) | \leq \frac{\sqrt 2}{\gamma} E_j(\omega) + \frac{\| f \|_{L^\infty(I)}}{\gamma^2} E_j(0).
\end{equation}

\textbf{Step 4: Well-definedness of the phases $\nu_j$.}
We use the upper bound on $E_j$ to show that the phases $\nu_j$ are well-defined. First, observe that
$$
|L_j(p_{j+1}-p_j)| \geq |f(p_{j+1})| - |f(p_{j+1}) - \tfrac{f(p_j)}{|f(p_j)|} L_j(p_{j+1}-p_j)|.
$$
Since $E_j(\omega) \leq \max_{1 \leq j \leq J-1} \| E_j \|_{L^\infty[0,p_{j+1}-p_j]} \leq \varepsilon$ for every $\omega \in [0,p_{j+1}-p_j]$ and since $\varepsilon$ satisfies the upper bound $\varepsilon \leq \min \left \{ \frac{\gamma^2}{2\sqrt{8}}, \frac{\gamma^3}{4 \| f \|_{L^\infty(I)}} \right \}$ it follows it follows from the upper bound obtained in \eqref{upper_bound_E} that 
$
    |f(p_{j+1}) - \tfrac{f(p_j)}{|f(p_j)|} L_j(p_{j+1}-p_j)| \leq \frac{\gamma}{2}.
$
In particular, this shows that
\begin{equation}\label{lower_bound_Lj}
    |L_j(p_{j+1}-p_j)| \geq \frac{\gamma}{2}.
\end{equation}
Thus, the phases $\nu_j$ are well-defined for all $j$.

\textbf{Step 5: Combining step 1-4.}

If $Q$ is given as above then from inequality \eqref{eq:E} we obtain
$$
Q \leq \frac{\sqrt 2 \max \{ 1, \| f \|_{L^\infty(I)}\}}{\min \{ \gamma, \gamma^2 \}} \underbrace{\max_{j = 1,\dots, J} \| E_j \|_{L^\infty} + E_j(0)}_{\coloneqq \Xi}
$$
Inequality \eqref{eq:rx_bound} and the lower bound \ref{lower_bound_Lj} implies that
\begin{equation*}
\begin{split}
& \min_{\tau \in \T} \| f - \tau \mathcal{R} \|_{L^\infty(I)}
\leq \frac{\sqrt 2 \max \{ 1, \| f \|_{L^\infty(I)}\} }{\min \{ \gamma, \gamma^2 \}} \Xi \\ & + \frac{4 \sqrt 2 (J-1) \max \{ \| f \|_{L^\infty(I)}, \| f \|_{L^\infty(I)}^2 \} }{\gamma \min \{ \gamma, \gamma^2 \}} \Xi 
+  \frac{8 (J-1) \max \{ 1, \| f \|_{L^\infty(I)}^2\} }{\gamma\min \{ \gamma^2, \gamma^4 \}} \Xi^2 \\
& \leq 8(J-1)\frac{\max \{ 1, \| f \|_{L^\infty(I)}, \| f \|_{L^\infty(I)}^2 \}}{\min \{ \gamma, \gamma^{2},\gamma^4,\gamma^5 \}} (2\Xi + \Xi^2)
\end{split}
\end{equation*}
In particular, the assumption on $\varepsilon$ and the fact that $\min \{ \gamma, \gamma^{2},\gamma^4,\gamma^5 \} = \min \{ \gamma,\gamma^5 \}$ and $\max \{ 1, \| f \|_{L^\infty(I)}, \| f \|_{L^\infty(I)}^2 \} = \max \{ 1, \| f \|_{L^\infty(I)}^2 \}$ gives
$$
\min_{\tau \in \T} \| f - \tau \mathcal{R} \|_{L^\infty(I)}  \leq 32(J-1)\frac{\max \{ 1, \| f \|_{L^\infty(I)}^2 \}}{\min \{ \gamma^1, \gamma^{5} \}} (\varepsilon + \varepsilon^2).
$$
\end{proof}

\subsection{Application to Gaussian shift-invariant spaces}\label{sec:atgss}

Recall that the inner products $\langle f_\omega, T_x\varphi_\omega \rangle$ were given in terms of the Fourier integral
\begin{equation}\label{eq:19}
\langle f_\omega, T_x\varphi_\omega \rangle = \int_\R |\mathcal{G}f(x,t)|^2 e^{-2\pi i \omega t} \, dt.
\end{equation}
To allow an approximation of functions in $f \in V_\beta^\infty(\varphi)$ from $|\mathcal{G}f(X)|$ where $X$ is a set of finitely many sampling points, we discretize the above integral using a suitable quadrature formula. If $\sigma$ is the standard deviation of the Gaussian $\varphi$ then we define $\sigma' \coloneqq \frac{1}{2\pi\sigma}$. The choice of a quadrature formula is based on the following factorization formula of the spectrogram.

\begin{theorem}\label{lma:spectrogram_factorization}
If $f \in V_\beta^\infty(\varphi)$ has defining sequence $c \in \ell^\infty(\Z)$ then the spectrogram of $f$ factors as 
\begin{equation}\label{eq:product_formula}
|\mathcal{G}f(x,t)|^2 = \pi \sigma^2 \varphi^{\sigma'}(t)S_x(t)
\end{equation}
where $S_x(t) = \sum_\ell b_\ell(x)e^{\pi i \beta \ell t}$ is a trigonometric series whose coefficients satisfy the Gaussian bound
$$
|b_\ell(x)| \leq \| c \|_\infty^2 \left ( 1 + \frac{\sigma}{\beta}\sqrt{2\pi} \right ) e^{-\frac{1}{8}(\frac{\beta \ell}{\sigma})^2}.
$$
In particular, for every $x \in \R$ the map $z \mapsto |\mathcal{G}f(x,z)|^2$ extends from $\R$ to an entire function on $\C$. The extension is given by the map $\C \ni z \mapsto \pi \sigma^2 \varphi^{\sigma'}(z)S_x(z)$.
\end{theorem}
\begin{proof}
See the Appendix \ref{appendix:D}.
\end{proof}

Recall that for $h>0$ and $H \in \N \cup \{ \infty \}$ the trapezoidal rule approximation of a map $W : \R \to \C$ is defined by
\begin{equation}\label{eq:trap_rule}
I_h^H(W) \coloneqq h \sum_{k=-H}^H W(hk).
\end{equation}
Applying this quadrature rule to the Fourier integral \eqref{eq:19} gives
\begin{equation}\label{q_applied}
    I_h^H(t \mapsto |\mathcal{G}f(x,t)|^2 e^{-2\pi i \omega t}) = h\sum_{k=-H}^H |\mathcal{G}f(x,hk)|^2e^{-2\pi i \omega hk},
\end{equation}
which essentially represents the discrete Fourier transform. We apply a classical result on the exponentially convergent trapezoidal rule \cite[Theorem 5.1]{trefethen}.

\begin{theorem}[Trefethen-Weideman]\label{thm:trefethen_w}
Let $a>0$ and suppose that $W$ is an analytic map in the strip $U_a = \{ z \in \C : |\mathrm{Im}(z)| <a \}$. Suppose further that $W(x) \to 0$ uniformly as $|x|\to \infty$ in the strip. If 
$$
M(W) \coloneqq \sup_{t+iy \in U_a} \int_{-\infty}^\infty |W(t+iy)| \, dt < \infty
$$
then, for any $h>0$, the trapezoidal rule approximation $I_h^\infty(W)$ of $\int_\R W$ exists and satisfies
$$
\left |I_h^\infty(W) - \int_\R W \right | \leq \frac{2M(W)}{e^{2\pi a/h}-1}.
$$
The quantity $2M$ in the numerator is optimal.
\end{theorem}

Truncating the series which corresponds to the biorthogonal expansion of $f_\omega$ and using the approximation \eqref{q_applied} of the inner products given in \eqref{eq:19}, results in the map
\begin{equation}\label{fct_S}
    S(\omega)= h \sum_{n=-N}^N \sum_{k=-H}^H \mathfrak{S}(n,k) e^{-2\pi i \omega hk} T_{\frac{\beta}{2}n} \widetilde{\varphi_\omega}(p+\omega), \ \ N \in \N,
\end{equation}
which is precisely the function $\sqrt{c_j}\overline{L_j}$ (with $p=p_j$) appearing in the error term $E_j$ defined in equation \eqref{eq:error_term}. The pointwise distance between the map $\omega \mapsto f_{\omega}(p+\omega)$ and $S$ can be upper bounded as follows.

\begin{lemma}\label{lma:error_bound}
Assume that $f \in V_\beta^\infty(\varphi)$ has defining sequence $c \in \ell^\infty(\Z)$.
Suppose that $N=\ceil{\frac{2}{\beta}(s+\frac{r}{2})}+m$ for some $m \in \N$ and some $r,s>0$. Let $h>0$ and $H \in \N$. Then for every $p \in [-s,s]$ and every $\omega \in [0,r]$ we have the error bound
\begin{equation}\label{ineq:22}
\begin{split}
& |f_\omega(p+\omega)-S(\omega)| \\
& \leq 4\sqrt{\pi} Ke^{\frac{\omega^2}{4\sigma^2}} \frac{\sigma}{\nu \beta} \left ( 1 + \frac{\sigma}{\beta}2\sqrt{2\pi} \right )^2 \| c \|_\infty^2 e^{-\frac{\nu\beta}{2}m} \\
        & + 2 \sqrt{\pi} Ke^{\frac{\omega^2}{4\sigma^2} + \frac{|\omega|}{\sigma}+1} \sigma \left ( 2 + \frac{4}{\nu\beta} \right ) \left ( 1 + \frac{\sigma}{\beta}2\sqrt{2\pi} \right )^2\| c \|_\infty^2 \frac{1}{e^{\frac{1}{\sigma h}}-1} \\
        & + \sqrt{\pi} Ke^{\frac{\omega^2}{4\sigma^2}} \sigma \left ( 2 + \frac{4}{\nu\beta} \right ) \left ( 1 + \frac{\sigma}{\beta}2\sqrt{2\pi} \right )^2 \| c \|_\infty^2 e^{-2(\pi H h \sigma)^2} \\
        & + \sqrt{2} Ke^{\frac{\omega^2}{4\sigma^2}} \left ( 2 + \frac{4}{\nu\beta} \right ) h \| \eta \|_\infty
\end{split}
\end{equation}
where $K=K(\sigma,\beta)$ and $\nu=\nu(\sigma,\beta)$ denote the decay constants of the dual generator as defined in \eqref{decay_dual_generator}.
\end{lemma}
\begin{proof}
See the Appendix \ref{appendix:E}.
\end{proof}

Lemma \ref{lma:error_bound} becomes meaningful if we interpret it in a qualitative way: recall that the constants $K$ and $\nu$ depend only on $\sigma$ and $\beta$. Suppose that $f$ satisfies condition \textbf{(P)} and that the variance $\sigma^2$ and the step-size $\beta$ are fixed, i.e. both the window function and the underlying signal space is fixed.
Under these assumption, it follows from inequality \eqref{ineq:22} and the definition of the error terms $E_j$ that
\begin{equation}\label{eq:qualitative_bound}
\max_{1 \leq j \leq J-1} \| E_j \|_{L^\infty[0,p_{j+1}-p_j]} \lesssim_{\sigma,\beta,r} \| c \|_\infty^2 e^{-am} + \| c \|_\infty^2 \frac{1}{e^{\frac{1}{\sigma h}}-1} + \| c \|_\infty^2 e^{-b(Hh)^2} + h \| \eta \|_\infty
\end{equation}
where $a$ and $b$ are constants depending only on $\sigma$ and $\beta$. The constant $r$ is given by
\begin{equation}
    r \coloneqq \max_{1 \leq j \leq J-1} (p_{j+1}-p_j).
\end{equation}
Let the universal constant in inequality \eqref{eq:qualitative_bound} be denoted by $\mathcal{D} = \mathcal{D}(\sigma,\beta,r)$. Then the following approximation result holds.

\begin{theorem}\label{thm:qualitative}
Assume that $f \in V_\beta^\infty(\varphi)$ has defining sequence $c \in \ell^\infty(\Z)$ and satisfies condition \textbf{(P)} with $p_1 < \cdots < p_J \in \R, -p_1=p_J=s$ and $\gamma>0$. Let $r \coloneqq \max_{1 \leq j \leq J-1} (p_{j+1}-p_j)$ and let $\varepsilon$ be chosen as in Theorem \ref{thm:abstract_bound}, i.e.
$
0 < \varepsilon \leq \min \left \{ \frac{\gamma^2}{2\sqrt{8}}, \frac{\gamma^3}{4 \| f \|_{L^\infty(I)}} \right \}.
$
Suppose that the sampling density $\frac{1}{h}$ in frequency direction satisfies
$$
\frac{1}{h} \geq \sigma \log\left (\frac{4\mathcal{D}\| c \|_\infty^2(J-1)}{\varepsilon}+1\right)
$$
and the measurements are given on the grid
$$
X = \tfrac{\beta}{2}\{-N, \dots, N \} \times h \{ -H,\dots,H \}
$$
such that the grid size characterizing parameters $N,H$ satisfy
$$
H \geq \frac{1}{h}\left ( \frac{1}{b} \log \left ( \frac{4\mathcal{D}\| c \|_\infty^2(J-1)}{\varepsilon} \right ) \right )^{\frac{1}{2}}
$$
and
$$
N \geq \ceil{\tfrac{2}{\beta}(s+\tfrac{r}{2})} +  \frac{1}{a} \log \left ( \frac{4\mathcal{D}\| c \|_\infty^2(J-1)}{\varepsilon} \right ).
$$
Then the following holds: if the noise level does not exceed $\frac{\varepsilon}{4h(J-1)\mathcal{D}}$, i.e. $\| \eta \|_\infty \leq \frac{\varepsilon}{4h(J-1)\mathcal{D}}$, then
\begin{equation}\label{error}
    \min_{\tau \in \T} \| f - \tau \mathcal{R} \|_{L^\infty[-s,s]}  \leq 32\frac{\max \{ 1, \| f \|^2_{L^\infty(I)} \}}{\min \{ \gamma, \gamma^5 \}} (\varepsilon + \varepsilon^2)
\end{equation}
where $\mathcal{R}$ denotes the numerical approximation routine.
\end{theorem}
\begin{proof}
Denote the four summands on the right-hand side of inequality \eqref{eq:qualitative_bound} by $\Xi_1, \dots, \Xi_4$. 
It follows that if $\mathcal{D}\Xi_i \leq \frac{\varepsilon}{4(J-1)}$ for $i=1,\dots,4$ then
$$
\max_{j = 1, \dots, J-1} \| E_j \|_{L^\infty[0,p_{j+1}-p_j]} \leq \frac{\varepsilon}{J-1}.
$$
Elementary manipulations show that the condition $\mathcal{D}\Xi_i \leq \frac{\varepsilon}{4(J-1)}$ for $i=1,\dots,4$ is satisfied by the above assumptions on $h,H,N$ and $\| \eta \|_\infty$. The statement is therefore a consequence of Theorem \ref{thm:abstract_bound}.
\end{proof}

Based on the previous statement we can study the growth of the number of samples needed to achieve the bound \eqref{error} for an increasing length of the reconstruction interval $I=[-s,s]$. Consider the noiseless case, $\eta = 0$. Suppose that condition \textbf{(P)} holds on the entire real-line, i.e. there exists a partition
$$
P = \cdots < p_{-1} < p_0 < p_1 < \cdots, \ |f(p_j)| \geq \gamma
$$
so that the maximal distance between two consecutive points $p_j$ and $p_{j+1}$ remains upper bounded by $r$, i.e.
$
\sup_{j \in\Z} (p_{j+1}-p_j) \leq r.
$
By going to a sub-partition we can always assume that for an interval $[-s,s]$ there exists $p_j,p_{k} \in P$ with $p_j \leq -s, p_k \geq s$ and $k-j \leq \frac{4s}{r}+\mathcal{O}(1)$.
Therefore, the assumptions on the bounds on $h,H$ and $N$ as stated in Theorem \ref{thm:qualitative} can be re-written as
$$
\frac{1}{h} \geq \sigma \log\left (\frac{\mathcal{D}'\| c \|_\infty^2s}{\varepsilon}+1\right), H \geq \frac{1}{h}\left ( \frac{1}{b} \log \left ( \frac{\mathcal{D}'\| c \|_\infty^2s}{\varepsilon} \right ) \right )^{\frac{1}{2}},
$$
$$
N \geq \ceil{\tfrac{2}{\beta}(s+\tfrac{r}{2})} +  \frac{1}{a} \log \left ( \frac{\mathcal{D}'\| c \|_\infty^2s}{\varepsilon} \right ).
$$
where $\mathcal{D}'$ is a constant depending only on $\sigma,\beta$ and $r$. A direct consequence is
\begin{corollary}\label{cor:upper_growth_estimate}
Under the assumptions above, the following holds: in order to achieve a reconstruction of $f \in V_\beta^\infty(\varphi)$ with defining sequence $c \in \ell^\infty(\Z)$ on an interval $I=[-s,s]$ of length $2s$ up an error bound of the form \eqref{error} using the numerical approximation routine $\mathcal{R}$, the number $\mathcal{N}(s)$ of spectrogram samples satisfies the upper growth estimate
$$
\mathcal{N}(s) \lesssim_{\sigma,\beta, r} \log \left ( \frac{s\| c \|_\infty^2}{\varepsilon} \right )^{\frac{3}{2}} \left ( s + \log \left ( \frac{s\| c \|_\infty^2}{\varepsilon} \right ) \right ). 
$$
\end{corollary}

Assumptions on $\sigma$ and $\beta$ turn the qualitative results above into quantitative results by means of explicit upper bounds on $a,b$ and $\mathcal{D}$. This can be done in a similar fashion as at the end of Section 3 and exploiting Lemma \ref{lma:expl_decay}.

\begin{proposition}
If $\frac{\beta}{4} \leq \sigma \leq \frac{\beta}{2} \leq 1$ and $r \leq 2 \sigma$ then
$$
\mathcal{D} \leq 40000 \left ( 2+ \frac{16}{\beta} \right ) \left ( 1+\frac{\sigma}{\beta} 2\sqrt{2\pi} \right )^2, \ \ a \geq \frac{\beta}{8}, \ \ b \geq 2\pi^2\sigma^2.
$$
\end{proposition}
\begin{proof}
This follows by applying the bounds on $K$ and $\nu$ derived in Lemma \ref{lma:expl_decay} to each term on the right-hand side of inequality \eqref{ineq:22}. 
\end{proof}

\subsection{Algorithm}

In a natural way, the foregoing results lead to a reconstruction algorithm whenever $f \in V_\beta^\infty(\varphi)$ satisfies condition \textbf{(P)}. This is an assumption on the signal itself which is unknown if only spectrogram samples are accessible. However, we have shown that the map $f_0=|f|^2$ can be recovered in a globally stably way by means of a biorthogonal expansion. Let the grid $X$ and the sampling set $\mathfrak{S}$ be given as above and set
$$
F = F_{N,H} = h \sum_{n=-N}^N \sum_{k=-H}^H \mathfrak{S}(n,k) T_{\frac{\beta}{2}n}\widetilde{\varphi_0}.
$$
If $\eta = 0$ then $F_{N,H} \to |f|^2$ uniformly on compact intervals. This motivates the following 4-step approximation procedure.


\begin{algorithm}[h]\label{algorithm:main}
\SetAlgoLined
\textbf{Input: $\mathfrak{S}=(|\mathcal{G}f(\tfrac{\beta}{2}n,hk)|^2+\eta_{n,k})_{n,k} \in \R^{(2N+1) \times (2H+1)}$, $ \tilde r,\tilde \gamma,s>0$}\;
\textbf{Output: $\mathcal{R} : [-s,s] \to \C$}
\begin{enumerate}
\setlength\itemsep{0em}
\item Set $F = h \sum_n \sum_k \mathfrak{S}(n,k) T_{\frac{\beta}{2}n}\widetilde{\varphi_0}$
\item Find $p_1 < \cdots < p_J \in I = [-s,s], J \geq 2$, such that $F(p_j) \geq \gamma$, $p_{j+1}-p_j\leq r$ and $p_{j+2}-p_j \geq r$ (if $J \geq 3$) for all $j \in \{ 1,\dots, J \}$
\item Define for $j \in \{ 1,\dots, J \}$ local reconstructive functions $L_j$ by
$$
L_j(\omega) = F(p_j)^{-\frac{1}{2}} h \sum_{n=-N}^N \sum_{k=-H}^H \mathfrak{S}(n,k)e^{2\pi i \omega h k} T_{\frac{\beta}{2}n} \widetilde{\varphi_\omega}(p_j+\omega)
$$
and phases $\nu_j$ by $$
\nu_0 = 1, \ \ \nu_j = \frac{L_j(p_{j+1}-p_j)}{|L_j(p_{j+1}-p_j)|} \  (j=2,\dots, J-1).
$$
\item Set $\mathcal{R}(t) = \nu_1 \cdots \nu_{j-1} L_j(\omega)$ if $t \in (p_j,p_{j+1}]$ with $t = p_j+\omega$ and $j=1, \dots, J-1$; $\mathcal{R}(p_1) = L_1(0)$ and $\mathcal{R}(t)=0$ otherwise
\end{enumerate} 
\caption{Reconstruction in $V_\beta^\infty(\varphi)$}
\end{algorithm}

Algorithm 1 approximates function in $V_\beta^\infty(\varphi)$ in a provably stable way.

\begin{theorem}\label{thm:numerical_sampling}
Suppose that $f \in V_\beta^\infty(\varphi)$ has defining sequence $c \in \ell^\infty(\Z)$ and $\| f \|_{L^\infty(\R)}, \| c \|_{\ell^\infty(\Z)} \leq \mathcal{L}$. Let $r,\gamma,s>0$ and
$$
0 < \varepsilon \leq \min \left \{ \frac{\gamma^2}{2\sqrt{8}}, \frac{\gamma^3}{4 \mathcal{L}} \right \}.
$$
Suppose that the samples are given on the grid $X$ with
$$
\frac{1}{h} \geq \sigma \log\left (\frac{16\mathcal{D}\mathcal{L}^2s}{\varepsilon r}+1\right), \
H \geq \frac{1}{h}\left ( \frac{1}{b} \log \left ( \frac{16\mathcal{D}\mathcal{L}^2s}{\varepsilon r} \right ) \right )^{\frac{1}{2}}, 
$$
$$
N \geq \ceil{\tfrac{2}{\beta}(s+\tfrac{r}{2})} +  \frac{1}{a} \log \left ( \frac{16\mathcal{D}\mathcal{L}^2s}{\varepsilon r} \right )
$$
and $\mathcal{D},a,b$ are defined as above. If the noise level satisfies $\| \eta \|_\infty \leq \frac{\varepsilon r}{16h\mathcal{D}s}$ then the following holds.
\begin{enumerate}\setlength{\itemsep}{0pt}
    \item Let $p_1 < \cdots < p_J \in I$ be the points detected in Step 2 of Algorithm 1 with input parameters $\tilde r = r, \tilde \gamma = \frac{3}{2}\gamma^2, \tilde s = s$. Then $p_1 < \cdots < p_J$ satisfy condition \textbf{(P)} with constants $r$ and $\gamma$
    \item If $p_1 < \cdots < p_J \in I$ are the points detected in Step 2 of Algorithm 1 with input parameters $\tilde r = r, \tilde \gamma = \frac{3}{2}\gamma^2, \tilde s = s$ and if $\mathcal{R}$ is the output function of Algorithm 1 then
    $$
    \min_{\tau \in \T} \| f - \tau \mathcal{R} \|_{L^\infty[p_1,p_J]}  \leq 32\frac{\max \{ 1, \mathcal{L}^2 \}}{\min \{ \gamma, \gamma^5 \}} (\varepsilon + \varepsilon^2)
    $$
\end{enumerate}
\end{theorem}
\begin{proof}
(1) If $p_1 < \cdots < p_J \in I=[-s,s]$ are the points detected in Step 2 of Algorithm 1 then the condition $p_{j+2}-p_j \geq r$ implies that $\frac{4s}{r}+1 \geq J$ and therefore $J-1 \leq \frac{4s}{r}$. Let $p \in [-s,s]$, let $S$ be defined as in equation \eqref{fct_S} and let $F$ be the function defined in Step 2 of Algorithm 1. Then $S(0)=F(p)$. Combining the inequalities \eqref{ineq:22} and \eqref{eq:qualitative_bound} with the assumptions on $h,H,N$ and arguing in an analogue fashion as in the proof of Theorem \ref{thm:qualitative} yields
\begin{equation*}
    \begin{split}
        ||f(p)|^2 - F(p)| & = |f_0(p)-S(0)| \\
        & \leq \mathcal{D}\left ( \| c \|_\infty^2 e^{-am} + \| c \|_\infty^2 \frac{1}{e^{\frac{1}{\sigma h}}-1} + \| c \|_\infty^2 e^{-b(Hh)^2} + h \| \eta \|_\infty \right ) \\
        & \leq \frac{r\varepsilon}{4s} \leq \varepsilon \leq \frac{1}{2}\gamma^2.
    \end{split}
\end{equation*}
Since $p$ was arbitrary, this estimate holds for every $p \in [-s,s]$. Let $p_j \in I$ be one of the detected points. Then $F(p_j) \geq \tilde \gamma = \frac{3}{2}\gamma^2$ and therefore
$$
|f(p_j)|^2 \geq F(p_j) - |F(p_j)-|f(p_j)|^2| \geq \frac{3}{2}\gamma^2 - \frac{1}{2}\gamma^2 = \gamma^2.
$$
Since $j \in \{ 1,\dots, J \}$ was arbitrary and $J \geq 2$, we conclude from the inequality above that the points $p_1 < \cdots < p_J$ satisfy condition \textbf{(P)}.

(2) According to Part 1, the points $p_1 < \cdots < p_J$ satisfy condition \textbf{(P)}. The map $\mathcal{R}$ defined in Step 4 of Algorithm 1 is precisely the numerical approximation routine as defined in Definition \ref{def:num_approx_routine} restricted to $[p_1, p_J]$. Observing that $J-1 \leq \frac{4s}{r}, [p_1, p_J] \subset [-s,s]$ and $\| f \|_{L^\infty(\R)}, \| c \|_{\ell^\infty(\Z)} \leq \mathcal{L}$ and consulting Theorem \ref{thm:qualitative} implies the assertion.
\end{proof}

\begin{figure}[h]
\label{fig:example_sync}
\centering
\hspace*{-1.9cm}
  \includegraphics[width=16cm]{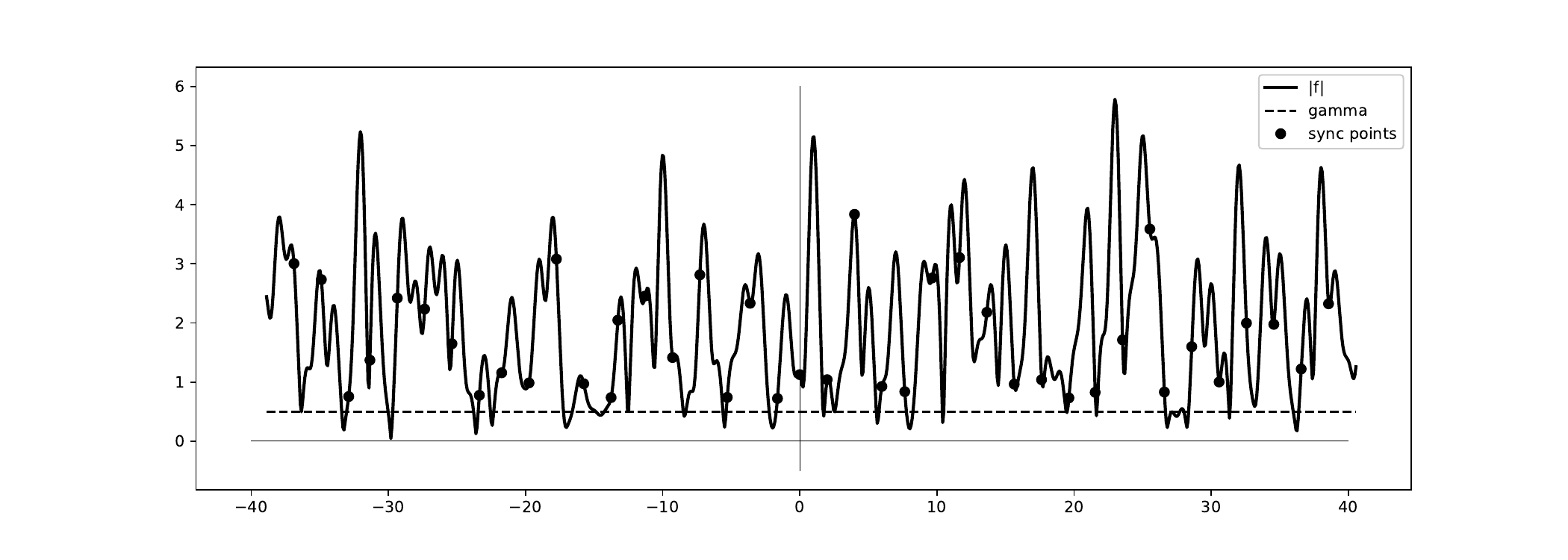}
\caption{Visualization of step 2 of Algorithm 1: plot of the absolute value $|f|$ of a function $f \in V_\beta^\infty(\varphi)$ on the interval $I=[-40,40]$. The black dots represent points $(p_j,F(p_j))$ with $p_{j+1}-p_j \leq r = 2$ and $F(p_j) \geq \gamma = 0.5$.}
\end{figure}

\begin{remark}
The partitioning step (Step 2) in Algorithm 1 can be executed as follows: evaluate the function $F$ on a grid $$-s = t_1 < \cdots t_L = s.$$ Select all points $t_{i_1} < \cdots < t_{i_M}$ such that $F(t_{i_m}) \geq \gamma$ for all $m= 1,\dots,M$. Remove those points $t_{i_m}$ of $t_{i_1} < \cdots t_{i_M}$ which satisfy
$$
t_{i_{m+1}}-t_{i_{m-1}} \leq r
$$
and do this as long as there are no points to remove anymore.
Call the resulting points $u_{i_1} < \cdots < u_{i_K}$. Let $i_k$ be the first index such that $u_{i_{k+1}}-u_{i_k} \leq r$ and $J \in \N$ the largest value such that $u_{i_{k+\ell}}-u_{i_{k+\ell-1}}\leq r$ for all $\ell = 1, \dots, J$. Finally, set $$p_\ell = u_{i_{k+\ell-1}}, \ \ell = 1, \dots, J.$$
\end{remark}

\begin{figure}\label{fig:example}
\centering
\hspace*{-1.9cm}
  \includegraphics[width=16cm]{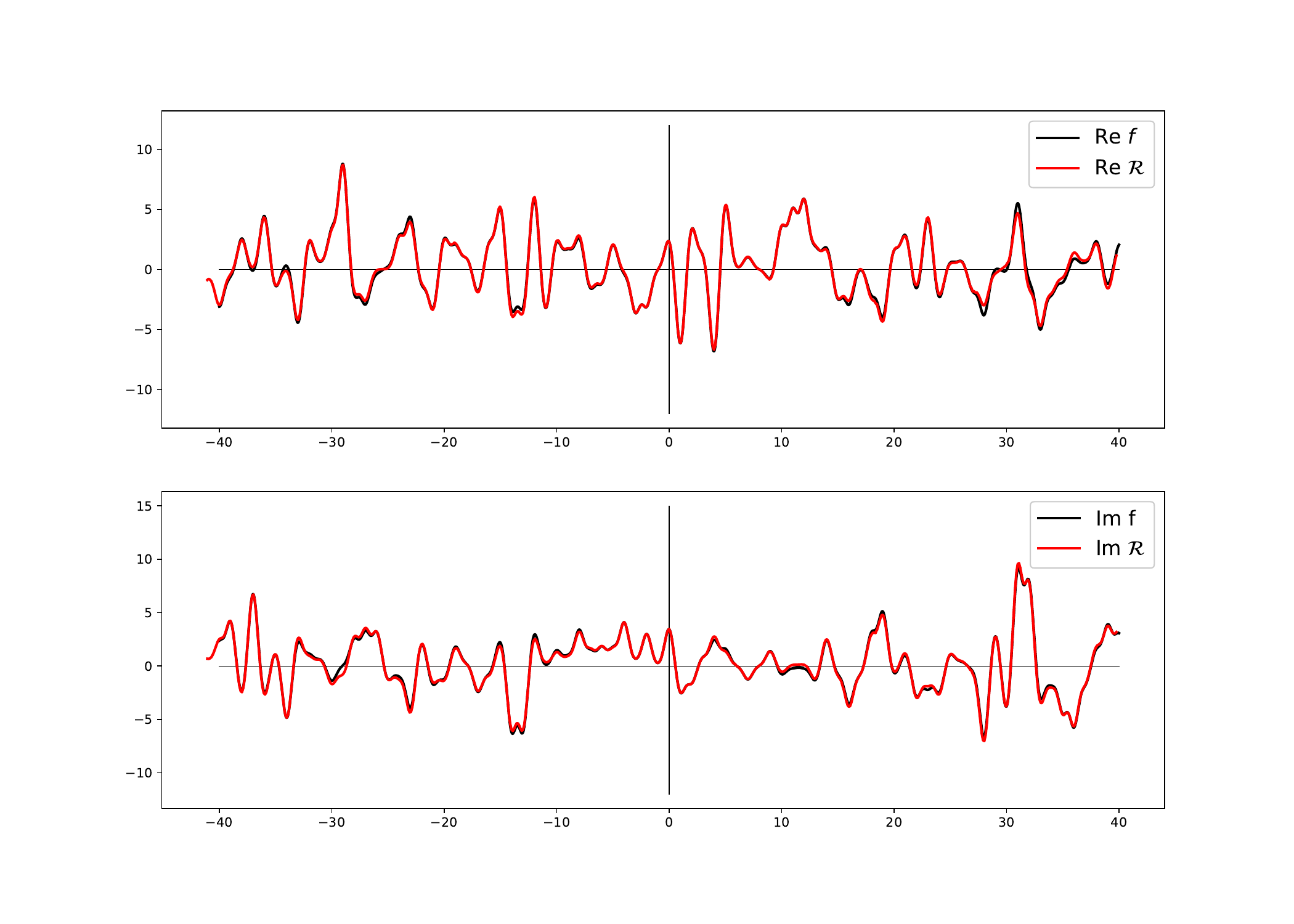}
\caption{Reconstruction of a complex-valued function $f \in V_\beta^\infty(\varphi)$ with $\beta = 1$ and generator $\varphi(t)=e^{-\pi t^2}$ on the interval $[-40,40]$ using Algorithm 1 with $r=1.5, \gamma=1$. The spectrogram samples are taken on a grid $X \subset \R^2$ of size $161 \times 151$, $X= \tfrac{1}{2} \{ -80, \dots, 80 \} \times \tfrac{1}{15}\{ -75, \dots, 75 \}$. To each spectrogram sample $|\mathcal{G}f(X)|^2$, Gaussian noise with mean zero and standard deviation 0.001 is added, i.e. the measurement matrix $\mathfrak{S}$ takes the form $\mathfrak{S} = |\mathcal{G}f(X)|^2+\mathcal{N}(0,\sigma^2), \sigma=0.001$. To visualize the approximation, we multiplied $\mathcal{R}$ with a global phase (the phase of $f$ at $p_1$, where $p_1$ arises from step 2 of Algorithm 1).}
\end{figure}

\begin{remark}
Suppose for simplicity that $\beta,\gamma,r \approx 1$ and that the noise is Gaussian of the form $\mathcal{N}(0,\sigma^2)$. Empirical simulations highlight that if the samples are taken on a grid which satisfies the conditions of Theorem \ref{thm:numerical_sampling} then Algorithm 1 reconstructs functions in $V_\beta^\infty(\varphi)$ with almost no visible error provided that the noise level satisfies $\sigma \leq 0.01$. If the noise level exceeds $0.01$ then the algorithm usually terminates earlier since no point $p \in \R$ with $|f(p)| \geq \gamma$ could be found.
\end{remark}

\begin{figure}
\centering
\hspace*{-1.9cm}
  \includegraphics[width=16cm]{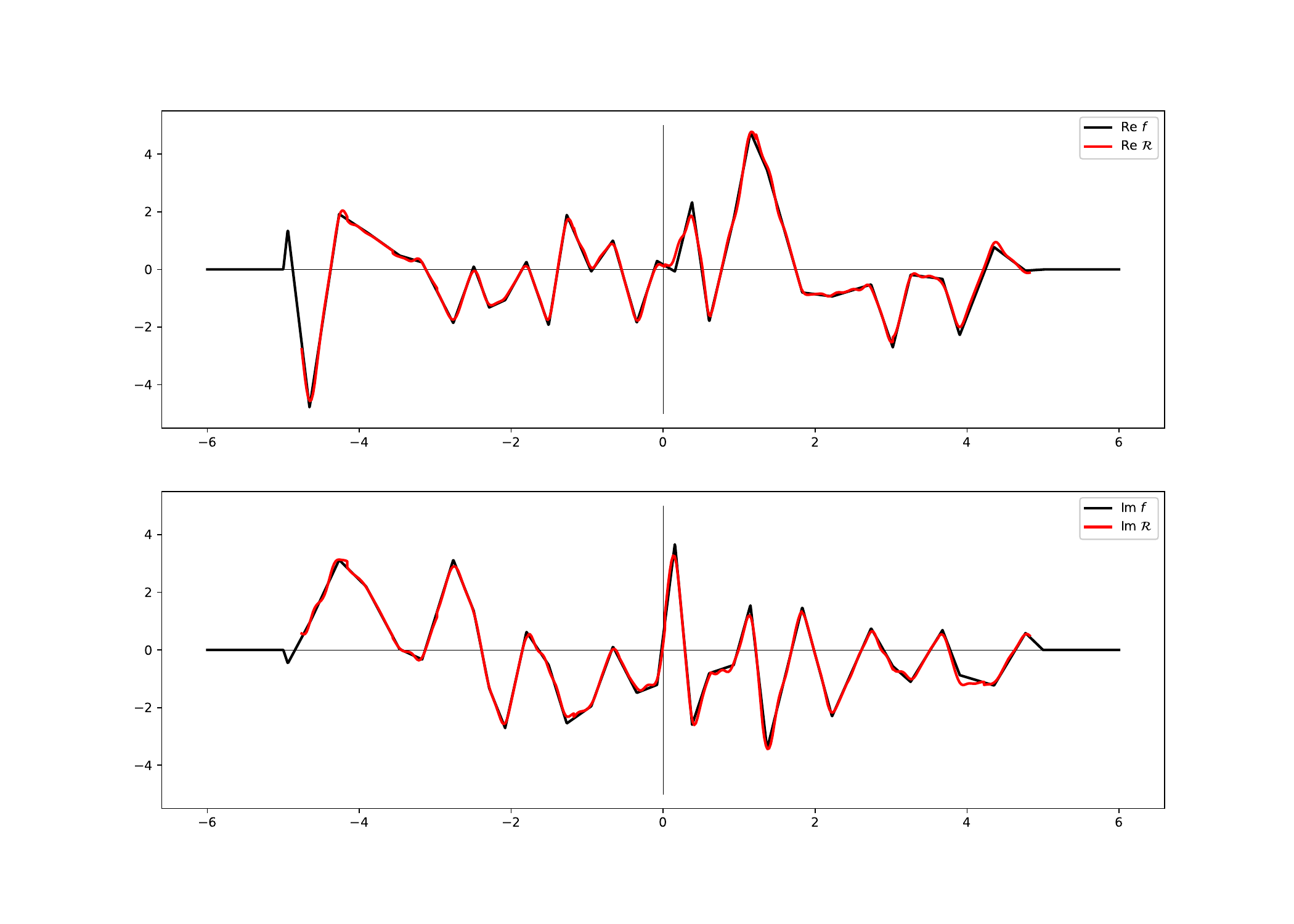}
\caption{Reconstruction of a complex-valued linear spline $f$ using Algorithm 1 with $\sigma=\beta=0.1, \gamma=0.2, r=0.3$ and no noise. The samples are taken on the grid $\frac{\beta}{2} \{ -120, \dots, 120 \} \times c \{ -150, \dots, 150 \}$. Since $f$ has support in $I=[-5,5]$ we choose $c = \frac{1}{2|I|} = \frac{1}{20}$. To visualize the approximation, we multiplied $\mathcal{R}$ with a global phase (the phase of $f$ at $p_1$, where $p_1$ arises from step 2 of Algorithm 1).}
\label{fig:example_pw_linear}
\end{figure}

\subsection{Beyond $V_\beta^\infty(\varphi)$}\label{sec:beyond_V}

\textbf{Compactly supported signals.}
Algorithm 1 was designed for functions in shift-invariant spaces with Gaussian generator. One could inquire about the justification of applying Algorithm 1 to different signal classes such as compactly supported maps. To give a meaningful answer to this question we define for $a,h>0$ the map $\mathcal{P}_\omega : L^4[-a,a] \to V_{\frac{\beta}{2}}^2(\varphi_\omega)$ via
$$
\mathcal{P}_\omega(f) = h \sum_{n\in\Z} \sum_{j\in\Z} |\mathcal{G}f(\tfrac{\beta}{2}n,hj)|^2 e^{2\pi i \omega h j} T_{\frac{\beta}{2}n}\widetilde{\varphi_\omega}.
$$
This is essentially the local reconstructive function defined in Section 4.1 with (noiseless) samples given on the infinite lattice $X=\frac{\beta}{2}\Z \times h \Z$.

\begin{proposition}
Let $a,h >0$ such that $ah \leq \frac{1}{4}$. Then for every $f \in L^4[-a,a]$ and every $\omega \in \R$ the map $\mathcal{P}_\omega(f)$ is the orthogonal projection of $f_\omega$ onto $V_{\frac{\beta}{2}}^2(\varphi_\omega)$.
\end{proposition}
\begin{proof}
General frame theory implies that the orthogonal projection $\mathcal{P}_\omega(g)$ of a map $g \in \lt$ onto $V_{\frac{\beta}{2}}^2(\varphi_\omega)$ is given by
$$
\mathcal{P}_\omega(g) = \sum_n \langle g,T_{\frac{\beta}{2}n}\varphi_\omega \rangle T_{\frac{\beta}{2}n}\widetilde{\varphi_\omega}.
$$
Since $f \in L^4[-a,a]$ we have $f_\omega \in L^2[-a,a]$.
The identity $\ft |\mathcal{G}f(x,\cdot)|^2(\omega) = \langle f_\omega, T_x \varphi_\omega \rangle$ implies that the Fourier transform of $|\mathcal{G}f(x,\cdot)|^2$ vanishes outside of the interval $[-2a,2a]$. 
Hence, the map $t \to |\mathcal{G}f(x,t)|^2$ is band-limited with bandwidth at most $4a$. The assumption on $h$ implies that $h \leq \frac{1}{4a}$ and Shannon's sampling theorem \cite[Theorem 6.13]{Higgins} yields
$$
|\mathcal{G}f(x,t)|^2 = \sum_j |\mathcal{G}f(x,hj)|^2 \mathrm{sinc}(\tfrac{1}{h}t - j), \ \ t \in \R.
$$
Setting $x=\frac{\beta}{2}n$ and applying the Fourier transform in the second argument of the spectrogram results in
$$
\ft (|\mathcal{G}f(\tfrac{\beta}{2}n,\cdot)|^2)(\omega) = \langle f_\omega , T_{\frac{\beta}{2}n} \varphi_\omega \rangle = h \sum_j |\mathcal{G}f(\tfrac{\beta}{2}n,hj)|^2 e^{2\pi i \omega h j}
$$
and this proves the statement.
\end{proof}

This result shows that if we run Algorithm 1 with a compactly supported function rather than a function in $V_\beta^\infty(\varphi)$ then in each step the algorithm approximates the projection of $f_\omega$ onto $V_{\frac{\beta}{2}}^2(\varphi_\omega)$. Heuristically, compactly supported functions that can be well-approximated by a linear combination of Gaussians and which satisfy condition \textbf{(P)} can be reconstructed with this methodology. In Figure \ref{fig:example_pw_linear} we visualize the performance of Algorithm 1 assuming that $f$ is a continuous function with compact support (but is not differentiable). 

\begin{figure}
\centering
\hspace*{-1cm}
   \includegraphics[width=14cm]{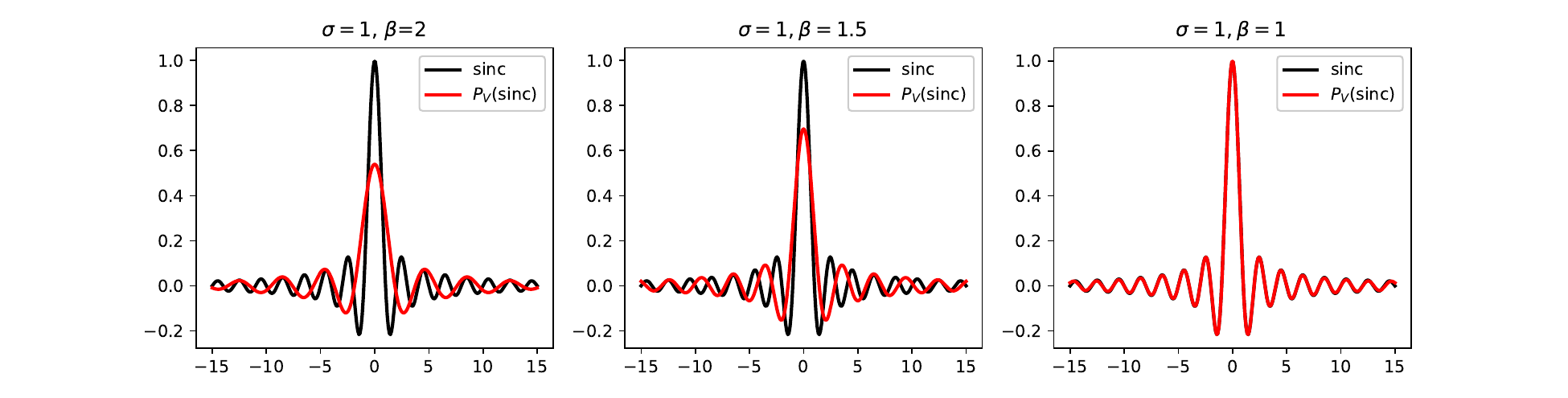}
\caption{Projection $P_V(\text{sinc})$ of a sinc function onto the shift-invariant space $V = V_\beta^2(\varphi^\sigma)$ with $\sigma=1$ and $\beta=1,\frac{3}{2},2$.}
\label{fig:proj}
\end{figure}

\textbf{Band-limited signals.} Suppose now that $f$ has compact support in the frequency domain, i.e. $f$ is a band-limited signal belonging to the Paley-Wiener space
$$
PW_a^2(\R) = \{ f \in L^2(\R) : \supp(\hat f) \subseteq [-a/2, a/2] \}, \ a>0.
$$
Empirically, Algorithm 1 approximates $f \in PW_a^2(\R)$ from noisy samples as soon as the standard deviation $\sigma$ of the Gaussian window $\varphi$ lies below a constant depending on the bandwidth $a$ of $f$. This follows from two observations. Firstly, $f$ is generated by a sinc function (Shannon's sampling theorem) and a sinc can be well-approximated by a linear combination of equally spaced Gaussians with standard deviation $\sigma \approx \frac{1}{a}$, see Figure \ref{fig:proj} for a visualization. Secondly, if $f$ has bandwidth $a$ then $f_\omega$ has bandwidth $2a$ for every $\omega \in \R$.
Hence, $PW_a^2(\R)$ exhibits a similar invariance of the tensor product operation as in the Gaussian shift-invariant setting (Proposition \ref{prop:invariance_tensor_product}).
Since $\varphi_\omega$ is (up to a constant) a Gaussian with standard deviation $\frac{\sigma}{\sqrt{2}}$ for every $\omega \in \R$, we observe that under the assumption $\sigma \approx \frac{1}{a}$ every $f_\omega$ can be well-approximated by projecting $f_\omega$ onto the Gaussian shift-invariant space $V_{\frac{\beta}{2}}^2(\varphi_\omega)$.
Clearly, general frame theory shows that the projection of $f_\omega$ onto $V_{\frac{\beta}{2}}^2(\varphi_\omega)$ can be derived as soon as the inner products $\langle f_\omega, T_x\varphi_\omega \rangle$ are available. Since
$$
\langle f_\omega, T_x\varphi_\omega \rangle = \int_\R |\mathcal{G}f(x,t)|^2e^{-2\pi i \omega t} \, dt
$$
(Proposition \ref{prop:fourier_identity}), the inner products can be recovered up to a small error using a suitable quadrature rule. 
In case of a band-limited signal $f \in PW_a^2(\R)$ we further observe that
$$
|\mathcal{G}f(x,t)| = |\ft(fT_x\varphi)(t)| \leq \left | \int_\R \hat f(t-s) \hat \varphi(s) \, ds \right |
$$
$$
\leq a \| \hat f \|_{L^2[-\frac{a}{2},\frac{a}{2}]} \| \hat \varphi \|_{L^\infty[-\frac{a}{2}+t, \frac{a}{2}+t]}
$$
where we used the convolution theorem for the Fourier transform. Thus, the map $t \mapsto |\mathcal{G}f(x,t)|$ satisfies a Gaussian decay as in Theorem \ref{lma:spectrogram_factorization}.
Figure \ref{fig:example_bandlimited} depicts the output of Algorithm 1 for an input function $f \in PW_a^2(\R)$ in a Gaussian noise regime.

\begin{figure}
\centering
\hspace*{-1.9cm}
  \includegraphics[width=16cm]{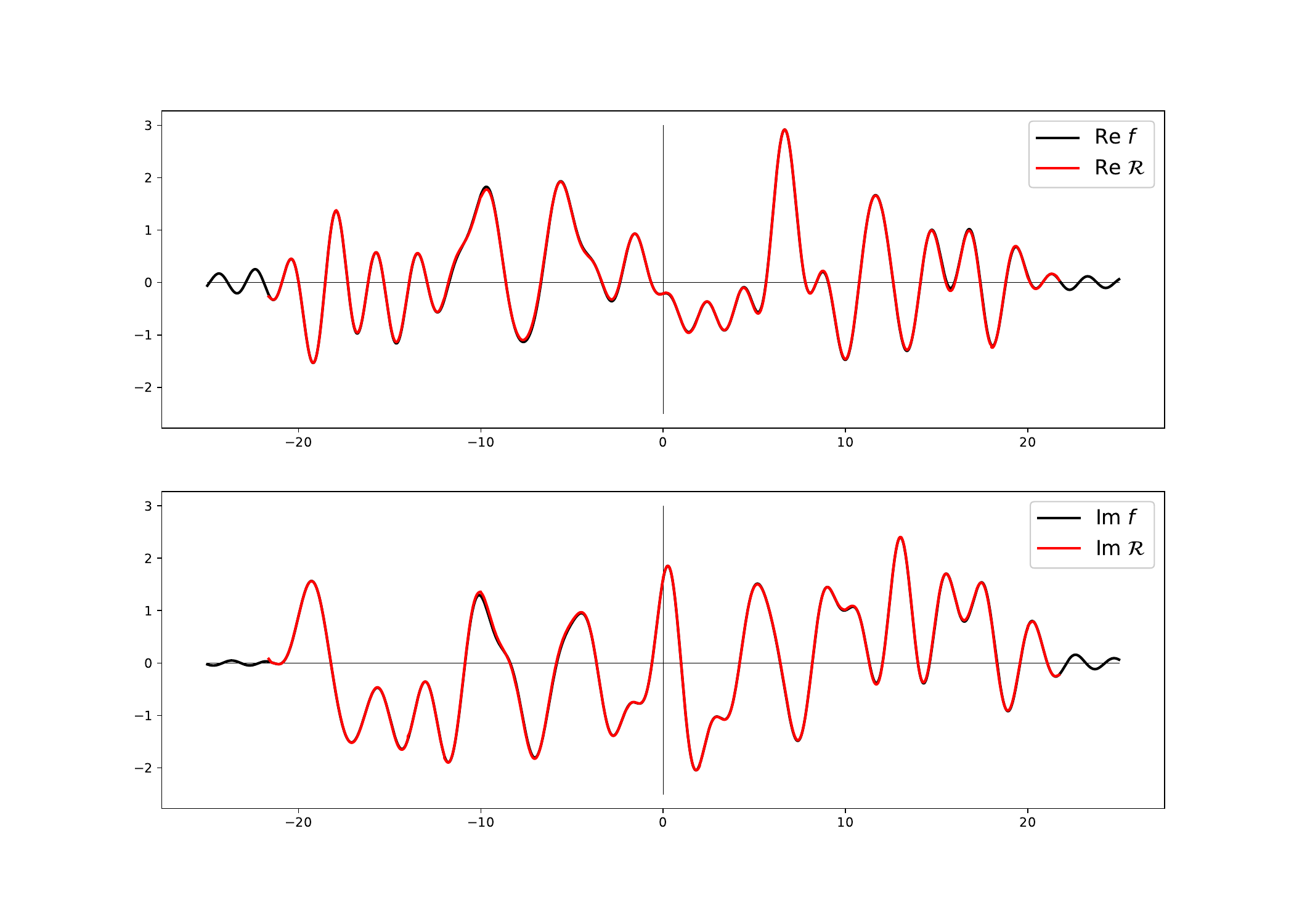}
\caption{Reconstructing a piece of a band-limited signal using Algorithm 1 with $\beta=0.5, r=1, \gamma=0.05$ and a Gaussian window with standard deviation $\frac{1}{4}$. The spectrogram samples are taken on the grid $X=0.04\{50,\dots,50\} \times \frac{\beta}{2} \{ -120,\dots,120\}$. To each spectrogram sample Gaussian noise $\mathcal{N}(0,\sigma^2)$ is added with $\sigma=0.0001$. To visualize the approximation, we multiplied $\mathcal{R}$ with a global phase (the phase of $f$ at $p_1$, where $p_1$ arises from step 2 of Algorithm 1).}
\label{fig:example_bandlimited}
\end{figure}

The foregoing discussion highlights that a combination of approximation properties of Gaussian shift-invariant spaces, numerical integration theory and the abstract theory presented in Section \ref{sec:reconstructive_function} point towards new approximation results from phaseless samples for function spaces beyond $V_\beta^\infty(\varphi)$. We leave this research direction open for future work.

\section{Appendix}

\subsection{Proof of Theorem \ref{thm:dual_generator}}\label{appendix:A}

\begin{proof}
The fact that $\theta$ is an element of $V_\beta^2(\phi)$ follows from \cite[Lemma 9.3.2]{christensenBook}. If $S$ is the frame operator of $(T_{\beta n}\phi)_n$ then
$$
\ft S \theta = \left ( \sum_{k \in \Z} \langle \hat \theta, M_{-k\beta}\hat \phi \rangle M_{-k\beta} \right ) \hat \phi.
$$
The inner products $\langle \hat \theta, M_{-k\beta}\hat \phi \rangle$ can be written as
\begin{equation*}
    \begin{split}
         \langle \hat \theta, M_{-k\beta}\hat \phi \rangle &= \int_\R \hat \theta(t)\overline{\hat \phi (t)} e^{2\pi i k\beta t} \, dt 
        = \sum_{n \in \Z} \int_0^{1/\beta} \hat \theta(t+\tfrac{n}{\beta})\overline{\hat \phi(t+\tfrac{n}{\beta})}e^{2\pi i k\beta t} \, dt \\
        &= \int_0^{1/\beta} \sum_{n \in \Z} \frac{\hat \phi(t+\frac{n}{\beta})}{\Psi_\beta(t+\frac{n}{\beta})} \overline{\hat \phi(t+\tfrac{n}{\beta})} \bone_{D \cap [0,\frac{1}{\beta}]}(t) e^{2\pi i k\beta t} \, dt \\ 
        &= \int_0^{1/\beta} \bone_{D \cap [0,\frac{1}{\beta}]}(t) e^{2\pi i k\beta t} \, dt
        = \frac{1}{\beta} \widehat{\bone_{D \cap [0,\frac{1}{\beta}]}}(k)
    \end{split}
\end{equation*}
where $\widehat{\bone_{D \cap [0,\frac{1}{\beta}]}}(k)$ is the $k$-th Fourier coefficient of $\bone_{D \cap [0,\frac{1}{\beta}]}$. By periodicity we have
$$
\sum_k \langle \hat \theta, M_{-k\beta}\hat \phi \rangle M_{-k\beta} = \bone_D.
$$
If $\hat \phi(t) \neq 0$ then $\Psi_\beta(t) \neq 0$ and therefore we have $\bone_D\hat \phi = \hat \phi$ which implies that $\ft S \theta = \frac{1}{\beta} \hat \phi$.
\end{proof}

\subsection{Proof of Proposition \ref{prop:V_2_V_inf}}\label{appendix:B}

\begin{proof}
Suppose that $f \in V_\beta^\infty(\varphi)$ has defining sequence $c = (c_n)_n \in \ell^\infty(\Z), f(t) = \sum_n c_n T_{\beta n} \varphi(t)$. Let $I\subset \R$ be a compact interval such that $I \subset [-s,s], s>0$.
Let $\varepsilon>0$. Define for $M \in \N$ the map $g_M$ by $g_M \coloneqq \sum_{m=-M}^M c_m T_{\beta m}\varphi$. Observe that the dual generator $\tilde \varphi$ converges to zero exponentially and since $c \in \ell^\infty(\Z)$ there exist a $C>0$ such that
$$
\left | \langle g_M, T_{\beta k} \varphi \rangle \right | \leq C
$$
for every $M \in \N$ and every $k \in \Z$. It follows that there exists an $L=L(s) \in \N$ such that
$$
\left | \sum_{\substack{k \in \Z \\ k \notin \{ -L,\dots,L \} }} \langle g_M, T_{\beta k} \rangle T_{\beta k} \tilde \varphi(t)  \right | \leq \varepsilon
$$
for every $t \in [-s,s]$ and every $M \in \N$. Now choose $M$ so that
$$
\left | \sum_{\substack{m \in \Z \\ m \notin \{ -M,\dots,M \} }} c_m T_{\beta m} \varphi(t) \right | \leq \varepsilon
$$
for every $t \in [-s,s]$. Since $g_M \in V_\beta^2(\varphi)$ we have the estimate
\begin{equation*}
    \begin{split}
        \left | f(t) - \sum_{n=-L}^L \langle f , T_{\beta n} \varphi \rangle T_{\beta n} \tilde \varphi(t) \right | 
         &\leq \varepsilon + \left | g_M(t) + \sum_{n=-L}^L \langle f , T_{\beta n} \varphi \rangle T_{\beta n} \tilde \varphi(t) \right | \\
        & \leq 2\varepsilon + \left |\sum_{n=-L}^L \langle g_M-f , T_{\beta n} \varphi \rangle T_{\beta n} \tilde \varphi(t) \right |.
    \end{split}
\end{equation*}
By dominated convergence, the right-hand side converges to $2\varepsilon$ as $M\to \infty$ uniformly in $t \in [-s,s]$.
\end{proof}

\subsection{Proof of Lemma \ref{lma:expl_decay}}\label{appendix:C}

\begin{proof}

\textbf{Step 1: Upper bounding $|\ift \Lambda|$.}
Let $c=\varphi(\frac{\beta}{2}) = e^{-\frac{\beta^2}{8\sigma^2}}$ and let
\begin{equation*}
    \begin{split}
        a_n & = (-1)^n \frac{2}{\xi} \sum_{m=0}^\infty (-1)^m c^{(m+\frac{1}{2})(2|n|+m+\frac{1}{2})},\\
        \xi & = \xi(c) = \sum_{n \in \mathbb Z} (-1)^n (2n+1)c^{(n+\frac 1 2)^2}
    \end{split}
\end{equation*}
be the Fourier coefficients of the reciprocal theta function as given in equation \eqref{theta_identity_2} and \eqref{theta_identity_3}. Then
$$
\ift \Lambda(t) = \sum_{n \in \Z} a_n T_{\frac{\beta}{2}n} \ift (s \mapsto e^{-\pi^2\sigma^2s^2})(t) = \frac{1}{\sigma\sqrt{\pi}} \sum_{n \in \Z} a_n e^{-\frac{(t+\frac{\beta}{2}n)^2}{\sigma^2}}.
$$
The modulus of the coefficients $|a_n|$ is upper bounded by
\begin{equation*}
    \begin{split}
        |a_n| & \leq \frac{2}{|\xi|} c^{|n|} \sum_{m=0}^\infty c^{(m+\frac{1}{2})^2} \leq \frac{2}{|\xi|} c^{|n|} \left ( c^{\frac{1}{4}} + \int_0^\infty c^{(s+\frac{1}{2})^2} \, ds \right ) \\
        & \leq \frac{2}{|\xi|} c^{|n|} \left ( c^{\frac{1}{4}} + \int_0^\infty c^{s^2} \, ds \right ) = \frac{2}{|\xi|} c^{|n|} \left ( c^{\frac{1}{4}} +  \frac{\sigma}{\beta} \sqrt{2\pi} \right ).
    \end{split}
\end{equation*}
It follows that
$$
|\ift \Lambda(t)| \leq \frac{2(c^{\frac{1}{4}} + \frac{\sigma}{\beta}\sqrt{2\pi})}{\sqrt{\pi}\sigma|\xi|} \sum_{n \in \Z} c^{|n|} e^{-\frac{(t+\frac{\beta}{2}n)^2}{\sigma^2}}
$$
where the right-hand side is even in $t$. In addition, the series on the right of the previous inequality satisfies
$$
\sum_{n \in \Z} c^{|n|} e^{-\frac{(t+\frac{\beta}{2}n)^2}{\sigma^2}} =\sum_{n \in \Z} c^{|n|} e^{-\frac{(|t|-\frac{\beta}{2}n)^2}{\sigma^2}} \leq \underbrace{\sum_{n = \ceil{\frac{1}{\beta}|t|}}^\infty c^{|n|}}_{\coloneqq A(t)} + \underbrace{\sum_{n=-\infty}^{\floor{\frac{1}{\beta}|t|}} e^{-\frac{(\frac{\beta}{2}n-|t|)^2}{\sigma^2}}}_{\coloneqq B(t)}.
$$
If $t \neq 0$ then the term $A(t)$ can be upper bounded via
$$
A(t) \leq \int_{\ceil{\frac{1}{2}|t|}-1}^\infty c^s \, ds = \frac{8\sigma^2}{\beta^2} e^{\frac{\beta^2}{8\sigma^2}} e^{-\frac{\beta^2}{16\sigma^2}|t|} \leq 2e^2e^{-\frac{1}{4}|t|}
$$
where in the rightmost inequality we used the assumption that $\frac{\beta}{4} \leq \sigma \leq \frac{\beta}{2} \leq 1$.
By making use of the elementary bound $e^{-bt^2} \leq e^\frac{b}{4}e^{-b|t|}$ ($t,b \in \R$) and the condition on $\sigma$ and $\beta$ we deduce the chain of inequalities
\begin{equation*}
    \begin{split}
        B(t) & \leq \sum_{n=-\infty}^{\floor{\frac{1}{\beta}|t|}-1} e^{-\frac{(\frac{\beta}{2}n-|t|)^2}{\sigma^2}} + e^{-\frac{(\frac{\beta}{2} \floor{\frac{1}{\beta}|t|} -|t|)^2}{\sigma^2}} 
         \leq \int_{-\infty}^{\floor{\frac{1}{\beta}|t|}} e^{-\frac{(\frac{\beta}{2}s-|t|)^2}{\sigma^2}} + e^{-\frac{t^2}{4\sigma^2}} \\
        & \leq \int_0^\infty e^{-\frac{(\frac{\beta}{2}s+\frac{1}{2}|t|)^2}{\sigma^2}} + e^{\frac{1}{16}}e^{-\frac{1}{4}|t|} 
         \leq \int_0^\infty e^{-\frac{1}{4}t^2} e^{-|t|s} e^{-s^2} \, ds + e^{\frac{1}{16}}e^{-\frac{1}{4}|t|} \\
        & \leq e^{\frac{1}{16}} e^{-\frac{1}{4}|t|} \int_0^\infty e^{-s^2} \, ds + e^{\frac{1}{16}}e^{-\frac{1}{4}|t|} 
         = \left (\frac{\sqrt \pi}{2} + 1 \right ) e^{\frac{1}{16}} e^{-\frac{1}{4}|t|}.
    \end{split}
\end{equation*}
Consequently, $|\ift \Lambda|$ satisfies the pointwise exponential bound
$$
|\ift \Lambda(t)| \leq \underbrace{ \frac{2(c^{\frac{1}{4}} + \frac{\sigma}{\beta}\sqrt{2\pi})}{\sqrt{\pi}} \left ( 2e^2+ \left ( \frac{\sqrt \pi}{2} + 1 \right)e^{\frac{1}{16}} \right) }_{\coloneqq \Xi} \frac{1}{\sigma |\xi|} e^{-\frac{1}{4}|t|}.
$$
If $\frac{\beta}{4} \leq \sigma \leq \frac{\beta}{2} \leq 1$ then 
$e^{-2} \leq c = \varphi(\frac{\beta}{2}) \leq e^{-\frac{1}{2}}$ and $\frac{\sigma}{\beta}\leq \frac{1}{2}$. Evaluating $\Xi$ results in the upper bound $\Xi \leq 41$.

\textbf{Step 2: Lower bounding $\xi(c)$.} 
For every $0<c<1$ we have
$$
\xi(c) = 2 \left ( \sum_{n=0}^\infty (4n+1)c^{\left (\frac{4n+1}{2}\right )^2} - \sum_{n=1}^\infty (4n-1)c^{\left (\frac{4n-1}{2} \right )^2} \right).
$$
Let $c \in J \coloneqq [e^{-2},e^{-\frac{1}{2}}]$ and define $f : \N \times J \to \R$ by $$f(n,c) = (4n-1)c^{\left (\frac{4n-1}{2} \right )^2}.$$ The map $f$ has the property that $n \mapsto f(n,c)$ is decreasing for every $c \in J$ and $c \mapsto f(n,c)$ is increasing for every $n \in \N$. Thus,
$$
\sum_{n=3}^\infty f(n,c) \leq \sum_{n=3}^\infty f(n,e^{-\frac{1}{2}}) \leq \int_2^\infty f(t,e^{-\frac{1}{2}}) \, dt = e^{-\frac{49}{8}}.
$$
It follows that for $c \in J$ we have
\begin{equation}
    \begin{split}
        \xi(c) & \geq 2 \left ( \sum_{n=0}^2 (4n+1)c^{\left (\frac{4n+1}{2}\right )^2} - f(1,c) - f(2,c) - \sum_{n=3}^\infty f(n,c) \right) \\
        & \geq 2 \left ( \sum_{n=0}^2 (4n+1)c^{\left (\frac{4n+1}{2}\right )^2} - f(1,c) - f(2,c) - e^{-\frac{49}{8}} \right) \\
        & = -2 e^{-\frac{49}{8}} + 2x -6x^9 +10x^{25} - 14 x^{49} + 18 x^{81} \eqqcolon p(x)
    \end{split}
\end{equation}
with $x = c^{\frac{1}{4}}$. For every $x \in [(e^{-2})^{\frac{1}{4}},(e^{-\frac{1}{2}})^{\frac{1}{4}}]$ the polynomial $p$ satisfies the lower bound $p(x) \geq \frac{1}{5}$. It follows that $\xi(c) \geq \frac{1}{5}$ for every $c \in J$.

\textbf{Step 3: Combining Step 1 and Step 2.}
The upper bound $\Xi \leq 41$ and the lower bound $\xi(c) \geq \frac{1}{5}$ for $c \in J$ implies that
\begin{equation}\label{eq:205_bound}
    |\ift \Lambda(t)| \leq \frac{205}{\sigma} e^{-\frac{1}{4}|t|}.
\end{equation}
Inequality \eqref{eq:205_bound} was derived under the condition that $t \neq 0$. By continuity it also holds for $t=0$. This concludes the proof of the Lemma.
\end{proof}

\subsection{Proof of Theorem \ref{lma:spectrogram_factorization}}\label{appendix:D}

\begin{proof}
Combining the identity
$
\mathcal{G}f(x,t) = \sum_k c_k e^{-2\pi i \beta k t} \mathcal{G}\varphi(x-\beta k, t)
$
with the identity
$$
\mathcal{G}\varphi(x,t) = \sigma \sqrt{\pi} e^{-\frac{x^2}{4\sigma^2}}e^{-\pi i x t} e^{-\pi^2\sigma^2t^2}
$$
yields
$$
|\mathcal{G}f(x,t)|^2 = \sum_k \sum_j \pi \sigma^2 c_k \overline{c_j} e^{2\pi i t(\tfrac{\beta}{2}(j-k))} e^{-2\pi^2\sigma^2t^2} e^{-\frac{(x-\beta k)^2}{4\sigma^2}} e^{-\frac{(x-\beta j)^2}{4\sigma^2}}.
$$
The relation
$$
e^{-\frac{(x-\beta k)^2}{4\sigma^2}} e^{-\frac{(x-\beta j)^2}{4\sigma^2}} = e^{-\frac{(x-\frac{\beta}{2}(k+j))^2}{2\sigma^2}} e^{-\frac{\beta^2(k-j)^2}{8\sigma^2}}
$$
shows that
$$
|\mathcal{G}f(x,t)|^2 = \pi \sigma^2 \sum_k\sum_j c_k \overline{c_j} a(\sigma,\beta,k,j) M_{\frac{\beta}{2}(j-k)} \varphi^{\sigma'}(t) T_{\frac{\beta}{2}(k+j)} \varphi^\sigma(x)
$$
where $a(\sigma,\beta,k,j) \coloneqq e^{-\frac{\beta^2(k-j)^2}{8\sigma^2}}$. The product formula \eqref{eq:product_formula} follows at once if we set
$$
b_\ell(x) = \sum_{j,k \in \Z, j-k=\ell} c_k \overline{c_j} e^{-\frac{\beta^2\ell^2}{8\sigma^2}} T_{\frac{\beta}{2}(k+j)} \varphi^\sigma(x)
$$
For every $\ell \in \Z$ and every $x \in \R$ we have the estimate
\begin{equation*}
\begin{split}
|b_\ell(x)| & \leq \| c \|_\infty^2 e^{-\frac{\beta^2\ell^2}{8\sigma^2}} \sum_j T_{\frac{\beta}{2}(2j-\ell)} \varphi^\sigma(x)  = \| c \|_\infty^2 e^{-\frac{\beta^2\ell^2}{8\sigma^2}} \sum_j T_{\beta j} \varphi^\sigma(x + \tfrac{\beta}{2}\ell) \\
& \leq \| c \|_\infty^2 e^{-\frac{\beta^2\ell^2}{8\sigma^2}} \sum_j T_{\beta j} \varphi^\sigma(0)  \leq \| c \|_\infty^2 e^{-\frac{\beta^2\ell^2}{8\sigma^2}} \left ( 1+\frac{1}{\beta} \int_\R \varphi^\sigma(t) \, dt \right ) \\
& = \| c \|_\infty^2 e^{-\frac{\beta^2\ell^2}{8\sigma^2}} \left ( 1+\frac{\sigma}{\beta}\sqrt{2\pi} \right )
\end{split}
\end{equation*}
where we used (similar to the proof of Corollary \ref{cor:tensor_stability}) that for real $x$ the theta function $x \mapsto \vartheta(x, c)$ is maximal at integers $\pi \Z$ \cite[p. 178]{JANSSEN1996165}.
In view of this Gaussian bound, the trigonometric series $S_x(z)$ is well-defined for $z \in \C$ and converges uniformly on compact subsets of $\C$. Hence, $z \mapsto |\mathcal{G}f(x,z)|^2$ extends from $\R$ to an entire function on $\C$.
\end{proof}

\subsection{Proof of Lemma \ref{lma:error_bound}}\label{appendix:E}

\begin{proof}
\textbf{Step 1: Numerical integration over $\R$.}
Suppose that $f \in V_\beta^\infty(\varphi)$ has defining sequence $c \in \ell^\infty(\Z)$. Let $S_x(t)$ be the trigonometric series as given in Theorem \ref{lma:spectrogram_factorization} so that
$$
|\mathcal{G}f(x,t)|^2 = \pi \sigma^2 \varphi^{\sigma'}(t)S_x(t).
$$
For $t,y,x,\omega \in \R$ let $W$ be defined by
$$
W(t+iy) \coloneqq |\mathcal{G}f(x,t+iy)|^2e^{-2\pi i \omega(t+iy)}.
$$
Further, let $A \coloneqq S_x(t+iy), B \coloneqq \varphi^{\sigma'}(t+iy)$ and $C \coloneqq e^{-2\pi i \omega (t+iy)}$. Then 
$$
W(t+iy) = \pi \sigma^2 ABC.
$$
We bound the modulus of $A,B$ and $C$ separately. Using the bound on $|b_\ell(x)|$ as given in Theorem \ref{lma:spectrogram_factorization} we obtain
\begin{equation*}
    \begin{split}
        |A| & = \left | \sum_\ell b_\ell(x) e^{\pi i \beta \ell(t+iy)} \right | \leq \sum_\ell |b_\ell(x)| e^{-\pi \beta \ell y} \\
         & \leq \| c \|_\infty^2 \left ( 1+\frac{\sigma}{\beta}\sqrt{2\pi} \right ) \sum_\ell e^{-\frac{1}{8}(\frac{\beta\ell}{\sigma})^2} e^{-\pi \beta \ell y} \\
         & = \| c \|_\infty^2 \left ( 1+\frac{\sigma}{\beta}\sqrt{2\pi} \right ) e^{2\sigma^2\pi^2y^2} \sum_\ell e^{-\frac{\beta^2}{8\sigma^2} \left ( \ell + \frac{4\sigma^2 y \pi}{\beta} \right )^2} \\
         & \leq \| c \|_\infty^2 \left ( 1+\frac{\sigma}{\beta}\sqrt{2\pi} \right ) e^{2\sigma^2\pi^2y^2} \sum_\ell e^{-\frac{\beta^2}{8\sigma^2} \ell^2} \\
         & \leq \| c \|_\infty^2 \left ( 1+\frac{\sigma}{\beta}\sqrt{2\pi} \right ) e^{2\sigma^2\pi^2y^2} \left ( 1+ \frac{1}{\beta}\int_\R e^{-\frac{s^2}{8\sigma^2}} \, ds \right ) \\
         & \leq \| c \|_\infty^2 \left ( 1+\frac{\sigma}{\beta}2\sqrt{2\pi} \right )^2 e^{2\sigma^2\pi^2y^2} 
    \end{split}
\end{equation*}
In addition, we have $|B| = \varphi^{\sigma'}(t)e^{\frac{y^2}{2(\sigma')^2}}$ and $|C| = e^{2\pi \omega y}$. It follows that
$$
|W(t+iy)| \leq \pi \sigma^2 \| c \|_\infty^2 \left ( 1+\frac{\sigma}{\beta}2\sqrt{2\pi} \right )^2 e^{2\sigma^2\pi^2y^2} e^{\frac{y^2}{2(\sigma')^2}} e^{2\pi \omega y} \varphi^{\sigma'}(t).
$$
Consider the strip $U = U_a = \{ z \in \C : |\mathrm{Im}(z)| <a \}$ with $a=\sigma'=\frac{1}{2\pi\sigma}$ being the conjugate standard deviation. It follows from the previous estimates that $W(z) \to 0$ uniformly as $|z|\to \infty$ in the strip $U$ and by Theorem \ref{lma:spectrogram_factorization} the map $W$ is analytic in $U$. 
Integration over the Gaussian $\varphi^{\sigma'}$ shows further that for every $t,y \in \R$
$$
\int_\R |W(t+iy)| \, dt \leq \sqrt{\frac{\pi}{2}} \sigma  \| c \|_\infty^2 \left ( 1+\frac{\sigma}{\beta}2\sqrt{2\pi} \right )^2 e^{2\sigma^2\pi^2y^2} e^{\frac{y^2}{2(\sigma')^2}} e^{2\pi \omega y}.
$$
From this we see that inside the strip $U$ the estimate 
$$
\sup_{t+iy \in U} \int_\R |W(t+iy)| \, dt \leq \sqrt{\frac{\pi}{2}} \sigma  \| c \|_\infty^2 \left ( 1+\frac{\sigma}{\beta}2\sqrt{2\pi} \right )^2 e^{\frac{|\omega|}{\sigma}+1}
$$
holds and the theorem of Trefethen and Weidemann, Theorem \ref{thm:trefethen_w}, implies that
$$
\left |I_h^\infty(W) - \int_\R W \right | \leq \sqrt{2\pi} \sigma  \| c \|_\infty^2 \left ( 1+\frac{\sigma}{\beta}2\sqrt{2\pi} \right )^2 e^{\frac{|\omega|}{\sigma}+1} \left ( \frac{1}{e^{\frac{1}{\sigma h}}-1} \right ).
$$

\textbf{Step 2: Estimating the cut-off error.}
For the integrand $W$ as defined above and an $H \in \N$ we estimate the cut-off error $| I_h^\infty(W) - I_h^H(W)|$. For every $\omega \in \R$ we have
\begin{equation*}
\begin{split}
| I_h^\infty(W) - I_h^H(W)| & \leq 2h \pi \sigma^2  \sum_\ell |b_\ell(x)|  \sum_{k=H+1}^\infty \varphi^{\sigma'}(hk) \\
& \leq 2h\pi\sigma^2 \| c \|_\infty^2 \left ( 1 + \frac{\sigma}{\beta}2\sqrt{2\pi} \right)^2 \sum_{k=H+1} \varphi^{\sigma'}(hk) \\
& \leq 2 \pi \sigma^2 \| c \|_\infty^2 \left ( 1 + \frac{\sigma}{\beta}2\sqrt{2\pi} \right)^2 \int_{Hh}^\infty e^{-2(\pi\sigma t)^2} \, dt.
\end{split}
\end{equation*}
Using a classical Gaussian tail bound of the form
$$
\int_q^\infty e^{-p t^2} \, dt \leq \sqrt{\frac{\pi}{4p}} e^{-pq^2}, \ \ \ p,q>0
$$
with $p=2\pi^2\sigma^2$ and $q=Hh$ results in
$$
| I_h^\infty(W) - I_h^H(W)| \leq \sqrt{\frac{\pi}{2}} \sigma \| c \|_\infty^2 \left ( 1 + \frac{\sigma}{\beta}2\sqrt{2\pi} \right)^2 e^{-2(\pi\sigma H h)^2}.
$$
and the triangle inequality implies that
$$
\left | \int_\R W - I_h^H(W) \right | \leq \sqrt{\frac{\pi}{2}} \sigma \left ( 1 + \frac{\sigma}{\beta}2\sqrt{2\pi} \right )^2\| c \|_\infty^2 \left ( \frac{2e^{\frac{|\omega|}{\sigma} + 1}}{e^{\frac{1}{\sigma h}}-1} + e^{-2(\pi Hh\sigma)^2} \right ).
$$

\textbf{Step 3. Estimating the distance $|f_\omega(p+\omega)-S(\omega)|$.}
Define for $H,N \in \N, h>0$ and $\omega \in \R$ the values $f(n,\omega)$ and $g(n,\omega)$ via
$$
f(n,\omega) \coloneqq \int_\R |\mathcal{G}f(\tfrac{\beta}{2}n,t)|^2e^{-2\pi i \omega t} \, dt
$$
and
$$
g(n,\omega) \coloneqq h \sum_{k=-H}^H \mathfrak{S}(n,k)e^{2\pi i \omega h k}.
$$
Using the biorthogonal expansion of $f_\omega$ and invoking the definition of the sampling set $\mathfrak{S}$ gives the upper bound
\begin{equation*}
    \begin{split}
        & |f_\omega(p+\omega) - S(\omega)| \\
        & \leq \underbrace{\left | \sum_{n \in ([-N,N]\cap\Z)^c} f(n,\omega) T_{\frac{\beta}{2}n}\widetilde{\varphi_\omega}(p+\omega) \right |}_{\coloneqq \varepsilon_1} + \underbrace{\left | \sum_{n=-N}^N (f(n,\omega)-\overline{g(n,\omega)}) T_{\frac{\beta}{2}n}\widetilde{\varphi_\omega}(p+\omega)  \right |}_{\coloneqq \varepsilon_2}.
    \end{split}
\end{equation*}
If $K$ and $\nu$ are the decay constants as given above then
\begin{equation*}
\begin{split}
\varepsilon_1 & \leq \sum_{n \in ([-N,N]\cap\Z)^c} |f(n,0)| | T_{\frac{\beta}{2}n}\widetilde{\varphi_\omega}(p+\omega)| \\
& \leq \sqrt{\frac{\pi}{2}} \sigma \left ( 1 + \frac{\sigma}{\beta}2\sqrt{2\pi} \right )^2\| c \|_\infty^2 \sqrt{2} e^{\frac{\omega^2}{4\sigma^2}} K \sum_{n \in ([-N,N]\cap\Z)^c} e^{-\nu |p+\frac{\omega}{2}-\frac{\beta}{2}n|}.
\end{split}
\end{equation*}
Since $N=\ceil{\frac{2}{\beta}(s+\frac{r}{2})}+m$ for some $m \in \N$, $p \in [-s,s]$ and $|\omega| \leq r$ it follows that for every $n \in ([-N,N]\cap\Z)^c$
$$
\left | p + \frac{\omega}{2} - \frac{\beta}{2}n \right | \geq \frac{\beta}{2} \left ( \ceil{\tfrac{2}{\beta}(s+\tfrac{r}{2})}+m \right ) - s - \frac{r}{2} \geq \frac{\beta}{2}m
$$
which implies that
$$
\sum_{n \in ([-N,N]\cap\Z)^c} e^{-\nu|a+\frac{\omega}{2}-\frac{\beta}{2}n|} \leq 2 \sum_{n=m+1}^\infty e^{-\frac{\nu\beta}{2}n} \leq 2\int_m^\infty e^{-\frac{\nu\beta}{2}t} \, dt = \frac{4}{\nu\beta} e^{-\frac{\nu\beta}{2}m}.
$$
Thus,
$$
\varepsilon_1 \leq \frac{4\sqrt{\pi}\sigma}{\nu\beta} \| c \|_\infty^2 \left ( 1 + \frac{\sigma}{\beta}2\sqrt{2\pi} \right)^2 K e^{\frac{\omega^2}{4\sigma^2}} e^{-\frac{\nu\beta}{2}m}.
$$
Now let $\mathcal{E} \coloneqq \left | \int_\R W - I_h^H(W) \right |$ be the quadrature error as derived in Step 2 above.
Then
\begin{equation*}
\begin{split}
\varepsilon_2 & \leq \mathcal{E} \sum_{n=-N}^N |T_{\frac{\beta}{2}n}\widetilde{\varphi_\omega}(p+\omega)| + \sum_{n=-N}^N \left ( h\sum_{k=-H}^H |\eta_{n,k}| \right ) |T_{\frac{\beta}{2}n}\widetilde{\varphi_\omega}(p+\omega)| \\
& \leq (\mathcal{E}+h\| \eta \|_\infty) \sum_{n=-N}^N |T_{\frac{\beta}{2}n}\widetilde{\varphi_\omega}(p+\omega)| \\
& \leq  (\mathcal{E}+h\| \eta \|_\infty) \sqrt{2}Ke^{\frac{\omega^2}{4\sigma^2}} \sum_{n \in \Z} e^{-\nu|p+\frac{\omega}{2} - \frac{\beta}{2}n|} \\
& \leq  (\mathcal{E}+h\| \eta \|_\infty) \sqrt{2}Ke^{\frac{\omega^2}{4\sigma^2}} \left ( 2 + \int_\R e^{-\frac{\nu\beta}{2}|t|} \, dt \right ) \\
& = (\mathcal{E}+h\| \eta \|_\infty) \sqrt{2}Ke^{\frac{\omega^2}{4\sigma^2}} \left ( 2 + \frac{4}{\nu\beta} \right )
\end{split}
\end{equation*}
and this yields the assertion by adding up $\varepsilon_1$ and $\varepsilon_2$.

\end{proof}

\bibliographystyle{acm}
\bibliography{bibfile}

\end{document}